\documentclass[a4paper,11pt,reqno]{amsart}

\RequirePackage[l2tabu, orthodox]{nag}
\usepackage[T1]{fontenc}
\usepackage[utf8]{inputenc}
\usepackage{lmodern}
\usepackage[frenchb,english]{babel}
\usepackage{color} \usepackage[usenames,dvipsnames]{xcolor}
\usepackage{amsthm,amssymb,amsmath}
\usepackage[pdftex]{graphicx}
\usepackage{enumerate}
\usepackage{tikz-cd}

\def\comment#1{}

\newcommand{\R}{\mathbb{R}}
\newcommand{\N}{\mathbb{N}}
\newcommand{\C}{\mathbb{C}}

\newcommand{\D}{\mathbb{D}}
\newcommand{\HH}{\mathbb{H}}
\newcommand{\PP}{\mathbb{P}}

\newcommand{\Ss}{\mathbb{S}}

\newcommand{\cF}{\mathcal{F}}
\newcommand{\cA}{\mathcal{A}}

\newcommand{\cL}{\mathcal{L}}

\newcommand{\cUL}{\mathcal{UL}}

\newcommand{\cRL}{\mathcal{RL}}

\newcommand{\cM}{\mathcal{M}}

\newcommand{\cC}{\mathcal{C}}
\newcommand{\cH}{\mathcal{H}}
\newcommand{\tL}{\widetilde{L}}

\def\to{\mathop{\rightarrow}}
\def\dans{\mathop{\subset}}
\newcommand{\moins}{\setminus}

\newcommand{\Iim}{\mathrm{Im}}

\newcommand{\dist}{\mathrm{dist}}

\newcommand{\Ciloc}{C^{\infty}_{\text{loc}}}

\newcommand{\Cil}{C^{\infty}_{l}}

\def\dans{\mathop{\subset}}

\newtheorem{theorem}{Théorème}[section]
\newtheorem{maintheorem}{Théorème}

\newtheorem*{theorem*}{Théorème}

\newtheorem{lemma}[theorem]{Lemme}
\newtheorem{proposition}[theorem]{Proposition}
\newtheorem{corollary}[theorem]{Corollaire}

\theoremstyle{definition}
\newtheorem{defn}[theorem]{Définition}
\theoremstyle{remark}
\newtheorem{rem}[theorem]{Remarque}
\newtheorem{ex}{Exemple}

\usepackage[margin=2.5cm]{geometry}

\usepackage{indentfirst}
\usepackage{microtype}

\usepackage[colorlinks=true,linkcolor=blue,citecolor=magenta]{hyperref}

\title[Prescription de courbure]{
Prescription de courbure des feuilles des laminations: retour sur un théorème de Candel
}

\date{}
\author[S. Alvarez]{S{\'e}bastien Alvarez}
\address{Sébastien Alvarez, CMAT, Facultad de Ciencias, Universidad de la Rep\'ublica, Uruguay}
\email{salvarez@cmat.edu.uy}

\author[G. Smith]{Graham Smith}
\address{Graham Smith, Instituto de Matem\'atica, UFRJ, Av. Athos da Silveira Ramos 149, Centro de Tecnologia - Bloco C, Cidade Universit\'aria - Ilha do Fund\~ao, Caixa Postal 68530, 21941-909, Rio de Janeiro, RJ - Brazil}
\email{grahamandrewsmith@gmail.com}

\begin{document}

\maketitle
\begin{abstract}
Dans cet article, nous revenons sur un célèbre théorème de Candel que nous renforçons en prouvant qu'étant donnée une lamination compacte par surfaces hyperboliques, toute fonction négative lisse dans les feuilles et transversalement continue est la fonction courbure d'une unique métrique laminée dans la classe conforme correspondante. Nous interprétons ce fait comme la continuité de solutions de certaines EDP elliptiques dans une topologie, dite de Cheeger-Gromov, sur l'espace des variétés riemanniennes complètes pointées.
 \end{abstract}

\begin{abstract}
In the present paper, we revisit a famous theorem by Candel that we generalize by proving that given a compact lamination by hyperbolic surfaces, every negative function smooth inside the leaves and transversally continuous is the curvature function of a unique laminated metric in the corresponding conformal class. We give an interpretation of this result as a continuity result about the solutions of some elliptic PDEs in the so called Cheeger-Gromov topology on the space of complete pointed riemannian manifolds.
\end{abstract}

\section{Introduction}

\paragraph{\textbf{Uniformisation simultanée}} Dans cet article, nous proposons de revenir sur un théorème célèbre, à juste titre, dû à Candel portant sur les \emph{laminations compactes par surfaces hyperboliques}: \cite{Candel}. Avant de l'énoncer, rappelons qu'une lamination compacte $(X,\cL)$ de dimension $d$ est la donnée d'un espace compact $X$  métrisable, muni d'un atlas de cartes laminées, c'est-à-dire localement homéomorphes au produit de $\R^d$ par un espace topologique, et dont les applications de transition préservent cette structure de produit local et sont lisses dans la direction $\R^d$. L'espace $X$ possède alors une partition en variétés de dimension $d$ appelées \emph{feuilles}. Lorsque $d=2$ et que les applications de transition entre cartes sont \emph{holomorphes} dans la direction $\R^2$, les feuilles sont naturellement munies de structures de surfaces de Riemann. Nous disons alors que $(X,\cL)$ est une \emph{lamination par surfaces de Riemann}. Si de plus les feuilles sont uniformisées par le disque, nous disons que $(X,\cL)$ est une lamination par surfaces hyperboliques.

Le théorème d'uniformisation simultanée de Candel s'énonce alors ainsi. \emph{Pour toute lamination compacte par surfaces hyperboliques $(X,\cL)$ il existe une unique famille de métriques riemanniennes complètes et conformes de courbure $-1$ dans les feuilles qui varie continûment transversalement dans les cartes au sens $\Ciloc$ (celui de la convergence uniforme des dérivées à tout ordre sur les compacts).} C'est un théorème subtil de continuité. L'existence d'une unique métrique hyperbolique complète conforme dans chaque feuille vient du théorème d'uniformisation classique de Poincaré-K{\oe}be. La difficulté est bien sûr de prouver que cette métrique varie continûment avec la feuille, et de donner un sens précis à ceci. Il faut également citer Verjovsky \cite{Ver}.

\paragraph{\textbf{Prescription de courbure}} Nous appellerons dans la suite \emph{lamination 
riemannienne} toute lamination munie d'une \emph{métrique laminée}, c'est-à-dire d'une famille de métriques dans les feuilles variant transversalement continûment dans les cartes au sens $\Ciloc$. Notre travail sera prétexte à développer un point de vue géométrique sur les laminations riemanniennes compactes déjà présent dans \cite{AL,AlvarezSmith,AlvarezBarral,ABC,AlvLopCandel,Lessa}. La philosophie sous-jacente à ces travaux est que la géométrie des feuilles d'une lamination riemannienne compacte possède une certaine récurrence. Plus précisément, les (revêtements des) feuilles d'une telle lamination forment une famille de variétés riemanniennes, compacte pour une certaine topologie, dite de \emph{Cheeger-Gromov} (définie au \S \ref{ss_chigouillegrougrouile}). Ce fait, prouvé par Lessa dans sa thèse (voir \cite{Lessa}) et dont nous donnerons une preuve élémentaire au \S \ref{ss.compacite}, nous permet d'exprimer un nouveau point de vue sur le théorème de Candel: il s'agit de \emph{la continuité des applications d'uniformisation pour la topologie de Cheeger-Gromov} sous des hypothèses que nous expliciterons au \S \ref{ss.cont_unif} (notons la similarité de ce résultat avec le problème de Plateau hyperbolique traité dans \cite{Smith_Plateau}). Il nous permet également de prouver le résultat principal de cet article.

\begin{maintheorem}\label{prescription_courboubouille}
Soient $(X,\cL)$ une lamination compacte par surfaces hyperboliques et $\kappa:X\to(0,\infty)$, une fonction positive lisse dans les feuilles et transversalement continue dans les cartes locales pour la topologie $\Ciloc$. Alors il existe une unique métrique laminée dans la classe conforme correspondante telle que la courbure gaussienne des feuilles soit en tout point donnée par $-\kappa$.
\end{maintheorem}

Ici encore, le problème n'est pas tant celui de l'existence d'une métrique complète de courbure négative prescrite dans chaque feuille (encore faut-il demander certaines bornes sur la courbure) que celui de la continuité transverse de cette métrique. Rappelons qu'étant donnée une surface riemannienne $(\Sigma,g)$, dont nous noterons $-\kappa_0$ la courbure gaussienne, et une fonction $u$ lisse sur $\Sigma$, la courbure gaussienne $-\kappa$ de la métrique riemannienne conformément équivalente $e^{2u}g$ est donnée par la formule
\begin{equation}\label{formule_magique}
-\kappa=e^{-2u}\left(-\kappa_0-\Delta u\right),
\end{equation}
où $\Delta$ représente l'opérateur de Laplace-Beltrami de la métrique $g$. Le problème de prescription de courbure se ramène donc à celui de la résolution de l'EDP elliptique d'inconnue $u$
\begin{equation}
\Delta u=e^{2u}\kappa-\kappa_0.
\end{equation}
Cette équation a une longue histoire. Poincaré lui-même a résolu une équation similaire pour uniformiser les surfaces fermées de caractéristique d'Euler négative: voir \cite{Poinc}. L'approche variationnelle moderne à ce problème est dûe à Berger dans \cite{Berger} qui prouve le Théorème \ref{prescription_courboubouille} dans le cas d'une surface fermée de caractéristique d'Euler négative. Kazdan et Warner vont plus loin dans \cite{KW} en caractérisant les fonctions qui peuvent être la fonction courbure d'une telle surface. Enfin dans \cite{AMc,BK,HT,KW2}, on trouvera l'étude de ce même problème de prescription de courbure pour des surfaces ouvertes. Nous donnerons nous-même au \S \ref{ss.prescription_disque} une preuve auto-contenue dans le cas des disques riemanniens complets à courbure négative, car nous aurons besoin de certaines bornes ainsi que de la régularité des solutions avec les données de l'équation.
 
Il s'agit donc d'étudier la continuité transverse de ces métriques, obtenues par résolution d'une EDP. Ce n'est pas la première fois qu'est posé le problème de la continuité transverse de la solution d'EDP dans les feuilles d'une lamination ou d'un feuilletage. La difficulté est toujours la même: la présence de l'holonomie impose des restrictions d'ordre dynamique, et fait en particulier que des feuilles voisines peuvent très bien ne pas être difféomorphes. Un exemple notable est le suivant. Garnett a introduit en 1983 les \emph{mesures harmoniques} (nous en parlerons au \S \ref{ss.gaussboubouille}) et a eu besoin de prouver la continuité transverse du noyau de la chaleur feuilleté: c'est \cite[Fact 1]{Garnett}. Dans cet article la continuité n'est prouvée qu'en supposant que les feuilles sont sans holonomie. C'est Candel qui en utilisant une méthode plus sophistiquée a prouvé le cas général dans \cite{Candel_harmonic}. C'est également ce type de problème qui a motivé le second travail de Lessa issu de sa thèse, \cite{Lessa2}. Enfin, alors qu'ils terminaient une première version de cet article, les auteurs ont pris connaissance d'un manuscrit de Mu\~niz et Verjovsky donnant une nouvelle preuve du théorème de Candel en utilisant le flot de Ricci tangent aux feuilles d'une lamination par surfaces hyperboliques: \cite{MunizVerjovsky}. 

Nous présentons ici deux nouvelles preuves du théorème de Candel: la première utilisant la continuité des applications d'uniformisation dans la topologie de Cheeger-Gromov, et la seconde basée sur une version laminée du théorème de Gauss-Bonnet prouvée par Ghys (voir le \S \ref{ss.gaussboubouille}) et sur la résolution d'EDP elliptiques.

\paragraph{\textbf{Prescription lisse}} Il serait intéressant d'avoir une version lisse du Théorème \ref{prescription_courboubouille}. Plus précisément, supposons qu'il nous soit donné un \emph{feuilletage} d'une \emph{variété} par surfaces hyperboliques, une métrique Riemannienne \emph{définie globalement}, ainsi qu'une fonction strictement négative, et que tous ces objets soient lisses. Par le Théorème \ref{prescription_courboubouille}, il existe une unique métrique laminée dans la classe conforme correspondante dont la courbure dans les feuilles soit donnée par cette fonction. \emph{Cette métrique laminée est-elle induite par une métrique riemannienne lisse définie globalement? En particulier, la métrique de Candel varie-t-elle de façon transversalement lisse avec la feuille?} 

Nos méthodes permettent de répondre positivement à cette question pour certains feuilletages qui sont ``développables'' en un certain sens: nous renvoyons au \S \ref{ss.variation_lisse} et plus précisément à l'énoncé du Théorème \ref{t.prescription_lisse_feuilletage}. C'est le cas par exemple des feuilletages transverses à des fibrations au dessus de surfaces hyperboliques compactes, et des feuilletages préservant certaines structures géométriques transverses, notamment les feuilletages riemanniens, et les feuilletages de Lie. Dans le cas général, le problème paraît nettement plus difficile que celui que nous traitons ici et requiert une étude dynamique de la séparation des feuilles.

\paragraph{\textbf{Motivations}} Le théorème de Candel a eu beaucoup d'impact, notamment en théorie des feuilletages et en systèmes dynamiques. Il est par exemple essentiel dans la théorie \emph{du cercle universel} développée par Thurston et qui relie l'existence d'actions de groupes fondamentaux de variétés de dimension $3$ sur le cercle à celle de feuilletages taut (ce sont les feuilletages munis d'une transversale fermée rencontrant toute feuille). En effet, ce théorème s'applique à toute variété $M$ atoroïdale de dimension $3$  munie d'un feuilletage taut $\cF$ par surfaces. Le groupe fondamental d'une telle variété agit alors sur ce que l'on appelle le cercle universel, obtenu en amalgamant les cercles à l'infini des feuilles de $\cF$. L'action obtenue reflète alors la dynamique du feuilletage: nous renvoyons à \cite{Calegari_book}, ainsi qu'à
 sa bibliographie très complète, pour une excellente introduction à cette théorie ``pseudo-Anosov'' des feuilletages.

D'autre part, l'étude des laminations par surfaces de Riemann a été popularisée par Sullivan et ses travaux en dynamique complexe: \cite{sullivan1988}. Il introduit dans cet article le concept d'\emph{espace de Teichmüller} d'une lamination par surfaces hyperboliques dont il approfondit l'étude dans \cite{sullivan1993}. Le théorème de Candel est fondamental dans les études faites de cet espace: voir notamment \cite{AL,AlvarezSmith,deroin2007,penner-saric2008,saric2009}. L'étude systématique de ces espaces est encore assez peu poussée, et nous espérons que le point de vue géométrique que nous développons puisse être utile.

Enfin, l'existence des métriques feuilletées ou laminées à courbure négative est importante en \emph{théorie ergodique des feuilletages}, étant entendu que ceux-ci ne préservent que très rarement des mesures transverses. Lorsque les feuilles d'un feuilletage sont à courbure négative, le \emph{flot géodésique feuilleté} possède une certaine forme d'hyperbolicité appelée \emph{hyperbolicité feuilletée}. Il est alors possible de profiter de cette hyperbolicité pour étudier des mesures invariantes intéressantes par ce flot (SRB, de Gibbs, etc.), comme cela est fait dans \cite{Alvharmonic,Alvarezu-Gibbs,AlvarezGibbs,AY,BGM}. Notons que le problème de trouver une \emph{entropie métrique} pour les feuilletages analogue à celle des systèmes dynamiques est encore ouvert. L'utilisation des métriques à courbure négative dans les feuilles paraît essentielle.

\paragraph{\textbf{Structure de l'article}} L'objectif principal de notre travail est d'introduire un cadre topologique adapté à l'étude de problèmes géométriques dans les feuilles de laminations compactes. Pour cela, nous aurons besoin de résultats déjà connus sous une forme légèrement différente (notamment sur la semi-continuité de la fonction feuille, voir \cite{Lessa} que nous interprétons comme un résultat de compacité, et sur les bornes des solutions de l'équation de courbure, voir \cite{AMc,BK}, ainsi que de leurs dérivées). Afin de simplifier la lecture de notre article, nous avons décidé de rédiger de nouvelles preuves, souvent plus courtes, de ces résultats.

Dans la section \ref{s.CG_unif} nous définissons le mode de convergence au sens de Cheeger-Gromov, et discutons la continuité des applications d'uniformisation. La section \ref{s.lam_CG} est dédiée à la preuve de la compacité de l'espace des revêtements riemanniens pointés des feuilles d'une lamination riemannienne compacte au sens de Cheeger-Gromov. La section \ref{simunif} contient la preuve du théorème d'uniformisation de Candel, ainsi que celle du Théorème \ref{prescription_courboubouille}. Dans la section \ref{s.presctiption}, nous montrons comment prescrire la courbure de disques riemanniens et étudions la dépendance des solutions de l'équation de courbure en fonction des données de l'équation. Enfin en appendice, nous donnons des preuves élémentaires de deux résultats déjà connus, mais fondamentaux pour notre étude. Il s'agit de la propriété de Hausdorff, à isométrie près, de la topologie de Cheeger-Gromov (Théorème \ref{CGHausdorff}) et de l'existence de métrique laminée de courbure négative (qui soit indépendante du théorème d'uniformisation simultanée) (Théorème \ref{negcurvedmetrics}).

\paragraph{\textbf{Remerciements}} C'est avec plaisir que nous remercions Fernando Alcalde Cuesta pour l'intérêt porté à ce travail, ainsi que pour des discussions sur la géométrie transverse des feuilletages. Nous remercions Jes\'us \'Alvarez Lopez, Pablo Lessa et Rafael Potrie pour d'intéressantes discussions. Enfin nous remercions Richard Mu\~niz et Alberto Verjovsky pour avoir partagé leur manuscrit avec nous. Enfin, c'est un plaisir de remercier le rapporteur pour ses remarques pertinentes qui ont permis d'améliorer le texte.  Durant la préparation de ce travail, S.A. a bénéficié de financement de la part de la ANII (via les projets FCE\_3\_2018\_1\_148740 et Mathamsud RGSD) et de la CSIC (via le projet I+D 389).

\section{Convergence au sens de Cheeger-Gromov et continuité de l'uniformisation}\label{s.CG_unif}

Le but de cette section est d'introduire la topologie de Cheeger-Gromov, ainsi que ses propriétés fondamentales que nous utiliserons dans la suite. Nous étudierons ensuite la continuité dans cette topologie des applications d'uniformisation de surfaces hyperboliques pointées complètes.

\subsection{Convergence au sens de Cheeger-Gromov}
\label{ss_chigouillegrougrouile}

\paragraph{\textbf{Définition }} Une \emph{variété riemannienne pointée} $(M,g,p)$ est la donnée d'une variété riemannienne complète $(M,g)$ et d'un point $p$ de $M$. Nous dirons qu'une suite $(M_n,g_n,p_n)_{n\in\N}$ de variétés riemanniennes pointées converge vers la variété riemannienne  pointée $(M,g,p)$ \emph{au sens de Cheeger-Gromov} s'il existe une suite d'applications $F_n:M_n\to M$ telles que
\begin{enumerate}
\item pour tout $n\in\N$, $F(p)=p_n$;

\noindent et telles que pour tout ouvert relativement compact $U\dans M$ il existe un entier $n_0>0$ vérifiant
\item pour tout $n\geq n_0$, la restriction de $F_n$ à $U$ est un difféomorphisme sur son image;
\item la suite des métriques $(F_n^{\ast}g_n)_{n\geq n_0}$ converge vers $g$ en restriction à $U$ au sens $\Ciloc$.
\end{enumerate}
Nous appellerons la suite $(F_n)$ une \emph{suite d'applications de convergence} de $(M_n,g_n,p_n)$ vers $(M,g,p)$. En général ce mode de convergence ne respecte pas les propriétés topologiques. Par exemple une limite de variétés simplement connexes, ou compactes, pourrait très bien ne pas l'être.

\paragraph{\textbf{\'Equivalence des applications de convergence}} Les suites d'applications de convergence ne sont bien sûr pas uniques. Mais deux suites $(F_n)$ et $(G_n)$ d'applications de convergence de $(M_n,g_n,p_n)$ vers $(M,g,p)$ sont \emph{\'equivalentes} au sens suivant. Il existe une suite $(\Phi_n)$ d'applications de $M$ vers elle-même, fixant $p$, ainsi qu'une \emph{isométrie} $\Phi:(M,p)\to(M,p)$ vérifiant que pour tout ouvert relativement compact  $U\dans M$ il existe un entier $n_1=n_1(U)>0$ tel que pour tout $n\geq n_1$
\begin{enumerate}
\item la restriction de $\Phi_n$ à $U$ est un difféomorphisme sur son image;
\item $G_n=F_n\circ \Phi_n$ en restriction à $U$;
\item $(\Phi_n)_{n\geq n_1}$ possède une sous-suite convergeant vers $\Phi$ en restriction à $U$ au sens $C^{\infty}_{\text{loc}}$.
\end{enumerate}
Cette équivalence est une conséquence du lemme suivant, que nous utiliserons plusieurs fois dans la suite.

\begin{lemma}
\label{equivalence_subtile}
Soient $M$ et $N$ deux variétés lisses de même dimension. Soient $(g_n)$ et $(h_n)$ deux suites de métriques riemanniennes sur $M$ et $N$ respectivement, qui convergent vers deux métriques $g$ et $h$ dans la topologie  $\Ciloc$. Soit $K\dans N$ une partie compacte et $\Phi_n:M\to K$ une suite d'applications lisses telles que pour tout $n\in\N$, $\Phi_n^\ast h_n=g_n$. Alors la suite $(\Phi_n)$ possède une sous-suite convergeant vers une application $\Phi:M\to K$ au sens $\Ciloc$ vérifiant de plus $\Phi^\ast h=g$.
\end{lemma}

Nous verrons en appendice la preuve de ce lemme ainsi que son corollaire: la propriété de Hausdorff (à isométrie près) de ce mode de convergence.

\begin{theorem}[Propriété de Hausdorff]\label{CGHausdorff} Soit $(M_n,g_n,p_n)$ une suite de variétés riemanniennes pointées qui converge vers $(M,g,p)$ et $(M',g',p')$ dans la topologie de Cheeger-Gromov. Alors il existe un difféomorphisme $\Phi:M\to M'$ tel que $\Phi^\ast g'=g$ et $\Phi(p)=p'$.
\end{theorem}

\subsection{Notions de convergence associées}
\label{ss.convergence}
\subsubsection{Convergence de suites d'applications au sens de Cheeger-Gromov}\label{sss.convergence(CG)}
Nous allons à présent définir une notion de convergence d'applications qui est cohérente avec la topologie de Cheeger-Gromov. Plus précisément, soit $(\Omega,h)$ une variété riemannienne et $q\in\Omega$. Soit $(M_n,g_n,p_n)$ une suite de variétés riemanniennes pointées convergeant vers $(M,g,p)$ au sens de Cheeger-Gromov. Nous prendrons une suite d'applications lisses $f_n:(\Omega,q)\to(M_n,p_n)$ et une application $f:(\Omega,q)\to(M,p)$. Bien sûr, si l'on veut définir la convergence de la suite $(f_n)$ vers $f$, il y a une difficulté venant du fait que les espaces d'arrivée de ces applications diffèrent.

\begin{defn}[Convergence d'applications au sens de Cheeger-Gromov]\label{def_appli_CG}
Nous disons que la suite d'applications lisses $f_n:(\Omega,q)\to(M_n,p_n)$ \emph{converge vers $f:(\Omega,q)\to(M,p)$ au sens de Cheeger-Gromov} s'il existe une suite d'applications de convergence $F_n:(M,g,p)\to (M_n,g_n,p_n)$, telle que pour tout ouvert relativement compact $O\dans M$, il existe un ouvert relativement compact $U\dans M$ ainsi qu'un entier $n_0>0$ tels que
\begin{enumerate}
\item $\overline{O}\dans U$;
\item pour tout $n\geq n_0$ la restriction de $F_n$ à $U$ est un difféomorphisme sur son image;
\item $f_n(f^{-1}(\overline{O}))\dans F_n(U)$;
\item la suite des restrictions $(F_n^{-1}\circ f_n|_{f^{-1}(O)})_{n\geq n_0}$ converge vers $f$ dans la topologie $C^\infty_{\text{loc}}$.
\end{enumerate}
\end{defn}

Nous résumons la définition précédente à l'aide du diagramme (non commutatif) suivant

\begin{equation*}
\begin{tikzcd}[row sep=huge]
f^{-1}(O)\dans(\Omega,q) \arrow[r,"f"] \arrow[d,swap,"f_n"] &
O\dans U\dans(M,p) \arrow[dl,"F_n"] &
\\
F_n(U)\dans(M_n,p_n) & & 
\end{tikzcd}
\end{equation*}

\begin{rem}
En utilisant la notion d'équivalence introduite au paragraphe précédent, il est facile de voir que la limite est unique à isométrie près: si $f_n:(\Omega,q)\to(M_n,p_n)$ converge vers $f:(\Omega,q)\to(M,p)$ et $f':(\Omega,q)\to (M',p')$ alors il existe une isométrie $\Phi:(M,g,p)\to (M',g',p')$ tel que $f'=\Phi\circ f$.
\end{rem}
\subsubsection{Suites de variétés peintes}\label{ss.peintures}

Une variété riemannienne pointée et \emph{peinte} est la donnée d'un quadruplet $(M,g,p,e)$, où $(M,g,p)$ est une variété riemannienne pointée et $e:M\to X$ est une fonction continue de $M$ vers un espace métrisable $X$. Nous disons alors que $M$ est peinte dans $X$. La suite de variétés pointées et peintes $(M_n,g_n,p_n,e_n)$ converge vers la variété pointée et peinte $(M,g,p,e)$ au sens de Cheeger-Gromov si $(M_n,g_n,p_n)$ converge vers $(M,g,p)$ et si, de plus, $e_n\circ F_n$ converge vers $e$ au sens $C^0_{\text{loc}}$ pour une certaine suite $(F_n)$ d'applications de convergence de $(M_n,g_n,p_n)$ vers $(M,g,p)$. Nous verrons à la section suivante que cette notion est particulièrement adaptée à l'étude des laminations.

\begin{rem}\label{rem.conv_appl_peint} Supposons que la suite $(M_n,g_n,p_n,e_n)$ converge vers $(M,g,p,e)$ et que la suite d'applications $f_n:(\Omega,q)\to(M_n,p_n)$ converge vers $f:(\Omega,q)\to(M,p)$ au sens de Cheeger-Gromov. L'égalité suivante $e_n\circ f_n=(e_n\circ F_n)\circ(F_n^{-1}\circ f_n)$ a lieu en restriction à tout compact, pour $n$ suffisamment grand. Nous en déduisons que $e\circ f_n$ converge vers $e\circ f$ dans la topologie $C^0_{\text{loc}}$.

\end{rem}

\subsubsection{Continuité de familles de métriques}\label{sss.continuite_metriques}
Soit à présent $\cM=(M_\alpha,g_\alpha,p_\alpha,e_\alpha)_{\alpha\in A}$ une famille de variétés riemanniennes pointées et peintes dans un espace métrisable $X$ qui soit compacte au sens de Cheeger-Gromov. Soit $h=(h_\alpha)_{\alpha\in A}$ une famille de métriques riemanniennes sur les éléments de $\cM$.

Nous disons que cette famille est \emph{continue au sens de Cheeger-Gromov} si pour toute suite $(M_{\alpha_n},g_{\alpha_n},p_{\alpha_n},e_{\alpha_n})$ d'éléments de $\cM$ convergeant vers $(M_\alpha,g_\alpha,p_\alpha,e_\alpha)\in\cM$, il existe une suite d'applications de convergence $(F_n)$ de $(M_{\alpha_n},g_{\alpha_n},p_{\alpha_n})$ vers $(M_\alpha,g_\alpha,p_\alpha)$ telle que
\begin{itemize}
\item $\lim_{n\to\infty} e_{\alpha_n}\circ F_n=e_\alpha$ au sens $C^0_{\text loc}$;
\item $\lim_{n\to\infty}F_n^\ast h_{\alpha_n}=h_\alpha$ au sens $\Ciloc$.
\end{itemize}

\begin{ex}
Cette notion de convergence est adaptée aux variétés \emph{peintes}. Nous allons illustrer la subtilité de cette définition par un exemple. Soit $X=\D\times(\{1/n;n\geq 1\}\cup\{0\})$ et la famille de variétés peintes $(\D,h_\D,0,e_n)$, telle que $e_n(z)=(z,1/n)$ pour $n\geq 1$ et $e_0(z)=(z,0)$ et où $h_\D$ représente la métrique de Poincaré sur $\D$. Soit $g$ une métrique riemannienne complète sur $\D$ telle qu'il n'y ait pas d'isométrie fixant $0$ autre que l'identité et soit $R$, une rotation quelconque, différente de l'identité. Soit $R_n$ une suite de rotations convergeant vers $R$ et $g_n$, une suite de métriques définies par $g_n=g$ pour $n\geq 1$ et $g_0=R^\ast g$. Alors
\begin{enumerate}
\item La suite $(R_n)$ est une suite d'applications de convergence de la variété riemannienne pointée $(D,h_\D,0)$ vers elle-même et $R_n^\ast g_n\to g_0$ quand $n$ tend vers l'infini au sens $\Ciloc$.
\item Il n'y a aucune suite de convergence $(F_n)$ de la variété riemannienne pointée $(D,h_\D,0)$ vers elle-même (nous utilisons ici le point ci-dessus, le Lemme \ref{equivalence_subtile} et le fait qu'il n'y ait pas d'isométrie non-triviale de $g$  fixant l'origine) telle que $F_n^\ast g_n\to g_0$ au sens $\Ciloc$ et $e_n\circ F_n\to e$.
\end{enumerate}
Ainsi, en oubliant la peinture, on pourrait dire que la famille $(g_n)$ est continue dans la topologie de Cheeger-Gromov mais elle ne l'est pas au sens où nous l'avons définie ci-dessus.
\end{ex}

\subsection{Continuité des applications d'uniformisation}\label{ss.cont_unif} 
Nous terminons cette section par un résultat de continuité des applications d'uniformisation de surfaces hyperboliques pointées: il s'agit du Théorème \ref{compact_unif}. Nous prouverons au \S\ref{ss.uniformouille} que ce résultat entraîne le théorème d'uniformisation simultanée de  Candel pour les laminations par surfaces hyperboliques.

Soit $(\D,h_\D)$ le disque unité muni de la métrique de Poincar\'e. Soit $\cM$ une famille de surfaces riemanniennes complètes pointées vérifiant trois conditions.

\begin{enumerate}
\item Les élements de $\cM$ sont conformément hyperboliques.
\item $\cM$ est compacte dans la topologie de Cheeger-Gromov.
\item $\cM$ est invariante par changement de point de base: si $(M,g,p)\in\cM$ et $q\in M$ alors $(M,g,q)\in\cM$.
\end{enumerate}
 
Par le théorème d'uniformisation, pour toute surface riemannienne pointée $(M,g,p)$ appartenant à $\cM$, il existe un revêtement holomorphe $\Pi:\Bbb{D}\rightarrow M$ tel que $\Pi(0)=p$. Un tel revêtement est uniquement déterminé à composition près par une rotation de $\D$, et sera appelé une \emph{application d'uniformisation} de $(M,g,p)$.

Dans cette section nous prouvons que l'application d'uniformisation varie continûment sur $\cM$. Plus précisément,
 
\begin{theorem}[Continuité de l'uniformisation]\label{compact_unif}
Soit $(M_n,g_n,p_n)$ une suite de surfaces pointées de $\cM$ convergeant vers $(M,g,p)$ au sens de Cheeger-Gromov. Pour tout $n\in\N$, considérons une application d'uniformisation $\Pi_n:\D\rightarrow M_n$. Alors il existe une application d'uniformisation $\Pi:\D\rightarrow M$ vers laquelle converge une sous-suite de $(\Pi_n)$ au sens de Cheeger-Gromov.
\end{theorem}

\begin{rem}\label{rem_unif-peinture}
Supposons de plus que les variétés de la famille $\cM$ soient peintes dans un certain espace métrisable $X$. Par la Remarque \ref{rem.conv_appl_peint},  si $(M_n,g_n,p_n,e_n)$ converge vers $(M,g,p,e)$ au sens de Cheeger-Gromov, alors il existe une application d'uniformisation $\Pi:\D\to M$ telle que $e_n\circ\Pi_n$ possède une sous-suite qui converge vers $e\circ\Pi$ au sens $C^0_{\text loc}$.
\end{rem}

Considérons à présent la famille $\widetilde{\cM}$ des revêtements universels d'éléments de $\cM$. Comme nous l'avons dit plus haut, la limite au sens de Cheeger-Gromov d'une suite de surfaces simplement connexes n'est pas nécessairement simplement connexe. Ainsi, la compacité de $\cM$ n'implique pas \emph{forcément} celle de $\widetilde{\cM}$. Cependant, le Théorème \ref{compact_unif} fournit immédiatement le résultat suivant.
\begin{corollary}\label{compact_univ_covers}
$\widetilde{\cM}$ est compacte pour la topologie de Cheeger-Gromov.
\end{corollary}

\begin{proof}[Preuve] Notons tout d'abord que, puisque $\cM$ est compacte au sens de Cheeger-Gromov, ses éléments sont à géométrie uniformément bornée. Il en va donc de même pour ceux de $\widetilde{\cM}$. Il en découle que $\widetilde{\cM}$ est précompact: toute suite d'éléments de $\widetilde{\cM}$ possède une sous-suite convergeant au sens de Cheeger-Gromov. Nous renvoyons à \cite{Ch,Gr,Lessa,Pet} pour ce fait.

Soit $(M_n,g_n,p_n)$ une suite de surfaces pointées appartenant à $\widetilde{\cM}$. Supposons qu'elle converge vers la surface pointée
 $(M,g,p)$. Pour tout $n\in\N$, soit $\Pi_n:(\D,0)\rightarrow (M_n,p_n)$ une application d'uniformisation. Pour tout $n\in\N$,
  $M_n$ s'identifie donc à $(\D,\Pi_n^\ast g_n,0)$. En utilisant le Théorème \ref{compact_unif},
   nous pouvons supposer qu'il existe une application d'uniformisation $\Pi:(\D,0)\to(M,p)$ telle que $(\Pi_n^
   \ast g_n)$ converge vers $\Pi^\ast g$ dans la topologie $\Ciloc$. Ainsi, 
    $(\D,\Pi_n^\ast g_n,0)_{n\in\Bbb{N}}$ converge vers $(\D,\Pi^\ast g,0)$ au sens de
     Cheeger-Gromov, les applications de convergence étant tout simplement données par l'identité. Nous concluons en utilisant la propriété de Hausdorff de la topologie de Cheeger-Gromov (cf. Théorème \ref{CGHausdorff}).
\end{proof}

\subsection{Préliminaires}

Nous commençons par introduire quelques outils et concepts qui seront utiles lors de la preuve du Théorème \ref{compact_unif}.

\paragraph{\textbf{Module conforme}} Un anneau conforme est une surface de Riemann ouverte et connexe dont le groupe fondamental est cyclique. Par uniformisation, un anneau conforme $A$ est ou bien de type parabolique, auquel cas il est conformément équivalent au produit $
S^1\times\R$, ou bien de type hyperbolique, auquel cas il est conformément équivalent au produit $S^1\times]0,M[$, où $M\in]0,\infty]$ est uniquement déterminé. Ce nombre $M:=M(A)$ s'appelle le \emph{module conforme} de $A$.

Il est bien connu que le module conforme est lié à la longueur extrémale maximale que nous définissons à présent. Soit $A$ un anneau conforme. Soit $\Gamma(A)$ l'ensemble des lacets continus et essentiels (c'est-à-dire homotopiquement non-triviaux) sur $A$. La \emph{longueur extrémale} de $\Gamma(A)$ par rapport à une métrique conforme $g$ sur $A$ est définie comme
$$
L(A,g) := \inf_{\gamma\in\Gamma(A)}l(\gamma,g),
$$
où nous notons $l(\gamma,g)$ la longueur de $\gamma$ pour la métrique $g$ (éventuellement infinie, si $\gamma$ n'est pas rectifiable). 

\begin{lemma}\label{l_conf-modulus} Le module conforme de $A$ est alors donné par la formule
$$
\frac{1}{M(A)} = \sup_{g}\frac{L(A,g)^2}{2\pi},
$$
où la borne supérieure est prise parmi toutes les métriques conformes d'aire $1$.
\end{lemma}

\begin{rem}
Il est à noter que la normalisation est différente de celle choisie par Ahlfors dans son livre \cite{Ahlfors2006}. Dans cette référence le module conforme de $A$ est donnée par $2\pi M(A)$.
\end{rem}

\begin{rem}\label{rem_mod} Il est connu (cf \cite{Ahlfors2006}) que lorsque le module conforme est fini, la borne supérieure est atteinte sur l'anneau $S^1\times]0,M[$ par la métrique suivante, et par elle seule
$$
g = \frac{1}{2\pi M}\big(d\theta^2 + dt^2\big).
$$
\end{rem}

\begin{lemma}\label{l_modulus}

Si $A_1\subseteq A_2$ sont deux anneaux conformes tels que $\Gamma(A_1)\subseteq\Gamma(A_2)$, alors $M(A_1)\leq M(A_2)$. De plus, si $M(A_2)$ est fini, l'égalité a lieu si et seulement si $A_1=A_2$.

\end{lemma}
%

\begin{proof}[Preuve] Il suffit de considérer le cas où $M(A_2)$ est fini. Nous pouvons supposer que $A_2=S^1\times ]0,M[$, où $M=M(A_2)$. Nous avons noté dans la Remarque \ref{rem_mod}, que la longueur extrémale maximale de $A_2$ est réalisée par la métrique 
$$
g = \frac{1}{2\pi M}(d\theta^2 + dt^2).
$$
Soit $a\leq 1$ l'aire de $A_1$ pour cette métrique. Puisque la restriction à $A_1$ de $g/a$ est conforme et d'aire $1$ nous avons,
$$
\frac{1}{M(A_1)} \geq \frac{L(A_1,g/a)^2}{2\pi} \geq \frac{L(A_2,g/a)^2}{2\pi} \geq \frac{L(A_2,g)^2}{2\pi} = \frac{1}{M(A_2)},
$$
d'où l'inégalité. L'égalité a évidemment lieu si $A_1=A_2$. Réciproquement, elle ne peut avoir lieu que si $a=1$ et si $g$ est la métrique produit sur $A_1$. Dans ce cas, en effet, $A_1$ est feuilleté par géodésiques fermées de $g$, c'est-à-dire par les cercles $(S^1\times\left\{t\right\})_{t\in]0,M[}$. Il en découle que $A_1=A_2$.
\end{proof}

\paragraph{\textbf{Lemme du quasi-maximum}} Nous continuons ces préliminaires avec un lemme de complétude, que nous croyons être dû à Gromov, et qui est à la fois élémentaire et très utile.

\begin{lemma}[Lemme du quasi-maximum]
\label{quasi-max}
Soit $(X,d)$ un espace métrique complet et $f:X\to[0,\infty)$ une fonction continue et positive. Alors pour tout $x_0\in X$, $C>1$ et $A,\alpha>0$ il existe $x\in X$ tel que $f(x)\geq f(x_0)$ et
$$d(x,y)\leq Af(x)^{-\alpha}\implies f(y)\leq Cf(x)$$
pour tout $y\in X$.
\end{lemma}

\begin{proof}[Preuve]
Supposons le contraire. Nous construisons par récurrence une suite $(x_n)$ de points de $X$ telle que pour tout $n\in\N$
$$d(x_n,x_{n+1})< Af(x_n)^{-\alpha}\,\,\,\,\,\,\,\&\,\,\,\,\,\,\, f(x_{n+1})> Cf(x_n).$$
Nous en déduisons que pour tout $n\geq 1$, $f(x_n)>C^nf(x_0)$ de sorte que $d(x_n,x_{n+1})<Af(x_0)^{-\alpha}C^{-n\alpha}$. Puisque $\alpha>0$ et $C>1$, la série $\sum_{n=0}^{\infty} C^{-n\alpha}$ est convergente et $(x_n)$ est de Cauchy. Par complétude de l'espace $X$, cette suite converge. Mais nous devons avoir $f(x_n)\to\infty$ quand $n$ tend vers l'infini. Cela contredit la continuité de $f$.
\end{proof}

\subsection{Discontinuité de l'uniformisation}

Avant d'entrer dans la preuve du théorème de continuité \ref{compact_unif} nous aimerions décrire un exemple instructif de discontinuité d'applications d'uniformisation.

\paragraph{\textbf{Une famille de surfaces de révolution}} Prenons comme modèle du disque hyperbolique la nappe d'hyperboloïde $\HH\dans\R^3$ d'équation $x^2+y^2-z^2=-1$ contenant $v_0=(0,0,1)$, munie de la métrique de Poincaré, notée $g_\HH$, c'est-à-dire la métrique hyperbolique obtenue par restriction de la métrique lorentzienne donnée par $ds^2=dx^2+dy^2-dz^2$. Notons $B_r\dans\HH$ le disque de rayon $r$ centré en $v_0=(0,0,1)$. C'est une surface de révolution autour de l'axe vertical, bordée par un cercle euclidien  $S_r$ de périmètre $2\pi\sinh(r)$.

Attachons à $S_r$ un cylindre euclidien $C_r$ dont la base a un périmètre $2\pi\sinh(r)$ et la hauteur est $r\sinh(r)$. Soit $C_r'\dans C_r$ un cylindre de même base que $C_r$, et dont la hauteur est $r\sinh(r)-1$. Lissons $B_r\cup C_r$ de manière symétrique par rapport aux rotations autour de l'axe vertical et notons $D_r$ la surface de révolution ainsi définie. Définissons alors dans $D_r$ une métrique $g_r$ satisfaisant aux hypothèses suivantes
\begin{itemize}
\item $g_r=g_\HH$ en restriction à $B_{r-1}$;
\item $g_r$ est conformément euclidienne, et invariante par rotation autour de l'axe vertical dans $C_r$;
\item $g_r$ est plate dans $C_r'$;
\item $g_r$ est complète.
\end{itemize}

\paragraph{\textbf{Convergence vers le plan de Poincaré}} Considérons la famille de surfaces riemanniennes pointées $(D_r,g_r,v_0)$ dont nous rappelons que le point de base est $v_0=(0,0,1)$. Il est immédiat que lorsque $r$ tend vers l'infini, cette famille de surfaces riemanniennes pointées converge vers le plan hyperbolique $(\HH,g_{\HH},v_0)$.

\paragraph{\textbf{Discontinuité de l'uniformisation}} Soit $\Pi_r:(\D,0)\to(D_r,v_0)$ une application d'uniformisation et $\Phi_r=\Pi_r^{-1}:(D_r,v_0)\to(\D,0)$. Le cylindre riemannien $(C_r,g_r)$ est conformément euclidien, symétrique par rotation autour de l'axe vertical et son module conforme vaut $r$ (voir le Lemme \ref{l_conf-modulus}) de sorte que
$$\Phi_r(C_r)=\left\{e^{-r}\leq |z|< 1\right\},$$
ce qui entraîne
$$\Phi_r(B_{r-1})\dans\left\{|z|<e^{-r}\right\}.$$

La dérivée en $v_0$ de l'application $\Phi_r$ tend donc vers zéro. D'où l'on déduit que la dérivée en $0$ de $\Pi_r$ explose et qu'il y a une discontinuité de l'uniformisation.

\paragraph{\textbf{Morale}} Que vient-il de se passer? Prenons $v_r$ n'importe quel point situé sur le cercle se trouvant à mi-hauteur du cylindre $C_r$. Quand $r$ tend vers l'infini, la famille de surfaces riemanniennes pointées $(D_r,g_r,v_r)$ converge vers le plan euclidien pointé $(\C,|dz|^2,0)$ au sens de Cheeger-Gromov: cette surface n'est bien évidemment pas hyperbolique. L'adhérence de la famille de disques pointés $\{(D_r,g_r,v)| r\geq 1, v\in D_r\}$ vérifie les deux dernières hypothèses que nous demandions pour la famille $\cM$ mais pas la première. Nous voyons donc que les hypothèses sur $\cM$ sont nécessaires.

\subsection{Preuve du Théorème \ref{compact_unif}}

Nous utiliserons un résultat élémentaire de la théorie de la régularité elliptique.

\begin{lemma}\label{l.elliptic_reg}

Soit $(\Omega,x_0)$ un domaine pointé du plan complexe $\C$. Soit $(M_n,g_n,p_n)$ une suite de surfaces pointées appartenant à $\cM$ et convergeant vers la surface riemannienne pointée $(M,g,p)$. Pour tout $n$, soit $f_n:\Omega\rightarrow M_n$ une fonction holomorphe telle que $f_n(x_0)=p_n$. Supposons que pour toute partie compacte $K$ de $\Omega$, il existe $C>0$ tel que, pour tout $x\in K$ et tout $n\in\N$,
$$
\|Df_n(x)\| \leq C.
$$
Alors il existe une fonction holomorphe $f:\Omega\rightarrow M$ telle que $f(x_0)=p$ vers laquelle converge une sous-suite de $(f_n)$ au sens de Cheeger-Gromov.

\end{lemma}

\begin{proof}[Preuve] Soient $x_0\in U_1\subseteq U_2$ deux ouverts relativement compacts de $\Omega$ tels que l'adhérence de $U_1$ soit incluse dans $U_2$. Supposons de plus que tout point de $U_2$ puisse être relié à $x_0$ par une courbe lisse incluse dans $U_2$ et de longueur inférieure à $R$, pour un certain $R>0$. Par hypothèse il existe $C>0$ tel que, pour tout $x\in U_2$ et tout $n\in\N$,
$$
\|Df_n(x)\| \leq C.
$$
Nous avons alors pour tout $n$,
$$
f_n(U_2)\subseteq B_{M_n}(p_n,CR).
$$
Soit $(F_n)$ une suite d'applications de convergence de $(M_n,g_n,p_n)$ vers $(M,g,p)$. Prenons $N\in\N$ tel que pour $n\geq N$, la restriction de $F_n$ à $B_{M}(p,2CR)$ soit un difféomorphisme sur son image, et dont l'image contient $B_{M_n}(p_n,CR)$. Quitte à augmenter $N$, nous pouvons supposer que pour tout $x\in U_2$ et tout $n\geq N$,
$$
\|D(F_n^{-1}\circ f_n)(x)\| \leq 2C.
$$
Par régularité elliptique, pour tout $k$, il existe $C_k>0$ tel que pour tout $x\in U_1$ et tout $n\geq N$,
$$
\|D^k(F_n^{-1}\circ f_n)(x)\| \leq C_k.
$$
L'ouvert $U_1$ étant arbitraire, nous pouvons conclure en utilisant le théorème d'Arzela-Ascoli.
\end{proof}

\begin{lemma}\label{l.upper_bound}

Soit $(M_n,g_n,p_n)$ une suite d'éléments de $\cM$. Pour tout $n\in\N$, soit $\Pi_n:\D\rightarrow M_n$ une application d'uniformisation. Il existe $C>0$ telle que pour tout $x\in\D$ et tout $n\in\N$,
$$
\|D\Pi_n(x)\| \leq C,
$$
la norme étant prise par rapport à la métrique de Poincaré $h_\D$.
\end{lemma}

\begin{proof}[Preuve] Supposons le contraire. Dans ce cas il existe une suite $(x_n)$ d'éléments de $\D$ tels que
$$
\lim_{n\to\infty}\|D\Pi_n(x_n)\| = \infty.
$$

Pour tout $n$, notons $R_n=\|D\Pi_n(x_n)\|$. Par le lemme du quasi-maximum nous pouvons de plus supposer que pour tout $n$ et tout $y\in B_{\D}(x_n,1/\sqrt{R_n})$, $\|D\Pi_n(y)\|\leq 2R_n.
$
Enfin, quitte à composer par une suite d'isométries nous pouvons supposer que $x_n=0$ pour tout $n$.

Considérons alors $H_{R_n}:\D_{R_n}=\{|z|<R_n\}\to\D$, l'homothétie  $z\mapsto z/R_n$, et définissons la métrique $g_{R_n}=R_n^2\,(H_{R_n}^{\ast}h_{\D})$. Cette métrique est complète sur $\D_{R_n}$ et converge vers $4|dz|^2$ sur toute partie compacte du plan dans la topologie $\Ciloc$. \'Etant donné $n\in\N$, nous définissons $\widetilde{\Pi}_n:B_{\D_{R_n}}(0,\sqrt{R_n})\rightarrow M_n$ en posant
$$
\widetilde{\Pi}_n(x) := \Pi_n\circ H_{R_n}(x)= \Pi_n\left(x/R_n\right).
$$
Notons que tout domaine borné $\Omega\subseteq\C$ est inclus dans $B_{\D_{R_n}}(0,\sqrt{R_n})$ à partir d'un certain rang. De plus pour tout $n$, $\|D\widetilde{\Pi}_n(0)\| = 1,$ et pour tout $x\in B_{\D_{R_n}}(0,\sqrt{R_n})$, $\|D\widetilde{\Pi}_n(x)\| \leq 2,$
la norme étant prise à la source par rapport à $g_{R_n}$. Pour tout $n$, posons $q_n:=f_n(0)$. Par hypothèse $(M_n,g_n,q_n)\in\cM$. De plus, par compacité de $\cM$, nous pouvons supposer que $(M_n,g_n,q_n)$ converge au sens de Cheeger-Gromov vers une surface riemannienne pointée $(M,g,q)\in\cM$, qui est par conséquent de type hyperbolique. Le Lemme \ref{l.elliptic_reg} nous fournit une application holomorphe $\widetilde{\Pi}:\C\rightarrow M$ vers laquelle converge une sous-suite de $(\widetilde{\Pi}_n)$ au sens de Cheeger-Gromov. 

En particulier, étant donné que $\|D\widetilde{\Pi}(0)\| = 1,$ la norme étant prise à la source par rapport à $4|dz|^2$, cette application est non-constante. Cela contredit le théorème de Liouville, et le lemme est prouvé.
\end{proof}

\begin{lemma}\label{l.lower_bound}
Soit $(M_n,g_n,p_n)$ une suite d'éléments de $\cM$. Pour tout $n\in\N$, soit $\Pi_n:\D\rightarrow M_n$ une application d'uniformisation. Il existe $C>0$ tel que, pour tout $x\in\D$ et tout $n\in\N$,
$$
\|D\Pi_n(x)\| \geq 1/C,
$$
la norme étant prise par rapport à la métrique de Poincaré $h_\D$.
\end{lemma}
\begin{proof}[Preuve] Supposons le contraire. Quitte à extraire et à composer par une suite d'isométries du disque de Poincaré, nous pouvons supposer que
$$
\lim_{n\rightarrow\infty}\|D\Pi_n(0)\| = 0.
$$
Pour tout $n\in\N$, posons $q_n=\Pi_n(0)$. Quitte à considérer les revêtements universels nous pouvons supposer que pour tout $n\in\N$, $M_n$ soit simplement connexe, et que l'application d'uniformisation soit un difféomorphisme. De plus, en extrayant de nouveau si besoin, nous pouvons supposer que la suite $(M_n,g_n,q_n)$ converge vers $(M,g,q)$.  Par compacité, cette limite est alors le revêtement riemannien d'un élément de $\cM$ (nous ne savons pas a priori si $M$ est simplement connexe). Fixons une fois pour toutes une suite $(F_n)$ d'applications de convergence. Par application des Lemmes \ref{l.elliptic_reg} et \ref{l.upper_bound}, nous pouvons supposer l'existence d'une application holomorphe $\Pi:(\D,0)\rightarrow (M,q)$ vers laquelle la suite $(\Pi_n)$ converge au sens de Cheeger-Gromov. En particulier, $D\Pi(0)=0$.
\par
Nous affirmons que $\Pi$ est constante. Supposons en effet qu'elle ne le soit pas. Puisque $\Pi$ est holomorphe il existe un voisinage $U$ relativement compact de $0$ dans $\D$ tel que $q\notin\Pi(\partial U)$ et $\deg(\Pi;q,U)>1$. Nous rappelons que $\deg(\Pi;q,U)$, le \emph{degré de $\Pi$ relativement à $U$}, est le nombre de préimages dans $U$, comptées avec multiplicité, d'une valeur régulière de $\Pi$ assez proche de $q$. Par définition de la convergence au sens de Cheeger-Gromov de suites d'applications (cf. la Définition \ref{def_appli_CG}) nous voyons que, lorsque $n$ est suffisamment grand, les propriétés suivantes sont vérifées
\begin{enumerate}
\item $F_n$ est un difféomorphisme d'un voisinage de $\Pi(U)$ sur son image ;
\item cette image contient $\Pi_n(U)$;
\item $q_n\notin\Pi_n(\partial U)$.
\end{enumerate} 
En appliquant la théorie classique du degré, nous déduisons que lorsque $n$ est suffisamment grand, $\deg(\Pi_n;q_n,U)=\deg(F_n^{-1}\circ\Pi_n;q,U)>1$, ce qui est absurde car $\Pi_n$ est un difféomorphisme. Nous déduisons donc que $\Pi$ est constante, confirmant l'affirmation.

Soit $B(q,r)$ le disque géodésique de rayon $r$ et de centre $q$ dans $M$, où $r$ est choisi inférieur au rayon d'injectivité de $M$ en $q$. Soit $\epsilon\in]0,r[$ et $A$, l'anneau $B(q,r)\setminus\overline{B(q,\epsilon)}$. Rappelons que les $F_n$ sont des applications de convergence. Par définition nous pouvons supposer que pour tout $n$ la restriction de  $F_n$ à $B(q,r)$ soit un difféomorphisme sur son image. De plus le module conforme de $F_n(A)$ tend vers celui de $A$ quand $n$ tend vers l'infini (par définition de la longueur extrémale).  

Soit à présent $\D_\rho$ le disque hyperbolique centré en $0$ et de rayon $\rho\in]0,+\infty[$. \'Etant donné que $(\Pi_n)$ converge au sens de Cheeger-Gromov vers $\Pi$, qui est constante, nous avons, pour $n$ suffisamment grand, $\Pi_n(\overline{\D_\rho})\dans F_n(B(q,\epsilon))$. Ainsi $F_n(A)$ est un anneau essentiel dans $\Pi_n(\D\moins\D_\rho)$. Nous déduisons, par le Lemme \ref{l_modulus}, que $M(A)=\lim_{n\to\infty}M(F_n(A))\leq\lim_{n\to\infty}M(\Pi_n(\D\moins\D_\rho))= M(\D\moins\D_\rho)$ (la dernière égalité vient de ce que $\Pi_n$ est holomorphe). C'est absurde car $M(\D\moins\D_\rho)$ tend vers $0$ lorsque $\rho$ tend vers l'infini.
 Ceci achève la preuve du lemme.
\end{proof}

\begin{proof}[Preuve du Théorème \ref{compact_unif}] Les Lemmes \ref{l.elliptic_reg}, \ref{l.upper_bound} et \ref{l.lower_bound} fournissent une application holomorphe $\Pi:\D\rightarrow M$ vérifiant $\Pi(0)=p$ et vers laquelle converge une sous-suite de $(\Pi_n)$ au sens de Cheeger-Gromov. Puisque la différentielle de cette application est uniformément bi-Lipschitz (par le Lemme \ref{l.lower_bound}), c'est un revêtement.
\end{proof}

\section{Laminations et espace des revêtements des feuilles}\label{s.lam_CG}

 Dans cette section nous introduisons la notion de lamination par variétés riemanniennes et établissons la compacité de l'espace formé par les revêtements riemanniens des feuilles d'une telle lamination.

\subsection{Laminations riemanniennes}
\paragraph{\textbf{Laminations} } Soit $X$ un espace métrisable compact. Une \emph{carte laminée} $d$-dimensionnelle de $X$ est un triplet $(U,T,\phi)$ où $U$ est un ouvert de $X$, $T$ est un espace métrisable, et $\phi:U\rightarrow(-1,1)^d\times T$ est un homéomorphisme. Un \emph{atlas laminé} sur $X$ est alors un recouvrement $(U_i,T_i,\phi_i)_{i\in I}$ de $X$ par cartes laminées telles tel que pour tous $i\neq j$, l'application de transition $\phi_j\circ\phi_i^{-1}$ s'écrit
$$
(\phi_j\circ\phi_i^{-1})(z,t) = (\zeta_{ij}(z,t),\tau_{ij}(t)),
$$
où $(z,t)$ est ici un élément de $\phi_i(U_i\cap U_j)$, et
\begin{enumerate}
\item $\tau_{ij}$ est un homéomorphisme entre un ouvert de $T_i$ et un ouvert de $T_j$;
\item à $t$ fixé, $\xi_{ij}(\cdot,t)$ est un difféomorphisme lisse sur son image;
\item $\xi_{ij}(\cdot,t)$ varie continûment dans la topologie $\Ciloc$ lorsque $t$ varie dans $T_i$.
\end{enumerate}
Deux atlas laminés sont \emph{équivalents} si leur union est encore un atlas laminé et nous dirons qu'une lamination $\cL$ de $X$ est une classe d'équivalence d'atlas laminés.

\'Etant donnée une carte laminée $(U,T,\phi)$ de $\cL$ et un point $t$ de $T$, l'ensemble $\phi^{-1}((-1,1)^d\times\left\{t\right\})$ s'appelle une \emph{plaque} de la carte. Les plaques se recollent, donnant ainsi une partition de $X$ en variétés lisses de dimension $d$, appelées les \emph{feuilles} de la lamination. Pour tout $x\in X$, la feuille passant par $x$ sera notée $L_x$. Chaque feuille $L$ vient avec un \emph{plongement canonique} $e:L\to X$.

\paragraph{\textbf{Fonctions lisses}}  Une fonction $u:X\rightarrow\R$ est dite de classe $C^k_l$ quand sa restriction à toute feuille est de classe $C^k$  et, de plus, varie continûment dans la topologie $C^k_\text{loc}$ avec le paramètre transverse dans toute carte laminée. Suivant Candel (cf. \cite{Candel}), l'espace de ces fonctions est noté $C^k_l(X)$. Nous munissons cet espace d'une structure naturelle d'espace de Banach de la façon suivante. Soit $(U_i,T_i,\phi_i)_{i\in I}$ un atlas laminé fini et soit $(\xi_i)_{i\in I}$ une partition de l'unité subordonnée à cet atlas qui soit de classe $\Cil$ (une telle partition existe toujours d'après \cite[Proposition $1.1$]{Candel}). La norme de Banach sur $C^k_l(X)$ est alors définie par
$$
\|u\|_{C^k_l} := \sum_{i\in I}\left\|(\xi_iu)\circ\phi_i^{-1}\right\|_{C^0(T_i,C^k((-1,1)^d))}.
$$
En particulier, deux normes construites de cette façon sont uniformément équivalentes. Nous avons alors le résultat suivant.

\begin{lemma}
\label{smoothdense}
Pout tout $k\geq 1$, l'espace $C^\infty_l(X)$ est dense dans $C^k_l(X)$.
\end{lemma}

\paragraph{\textbf{Laminations riemanniennes}} Soit $\cL$ une lamination d'un espace métrisable compact $X$ et supposons que
\begin{itemize}
\item chaque feuille $L$ est munie d'une metrique riemannienne de classe $C^{\infty}$ notée $g_L$;
\item la métrique $g_L$ varie transversalement continûment dans les cartes pour la topologie $\Ciloc$.
\end{itemize}
La correspondance $g:L\mapsto g_L$ s'appelle une \emph{métrique laminée} lisse, nous disons qu'elle est de classe $C^\infty_l$. La donnée de $(X,\cL,g)$ s'appelle une \emph{lamination riemannienne}.

\begin{rem}
Nous n'entendons pas ici \emph{riemannien} dans son acception classique en théorie des feuilletages. Un \emph{feuilletage riemannien} désigne en général un feuilletage dont les applications d'holonomie préservent une certaine métrique transverse: voir le \S \ref{ss.variation_lisse}. Ici, les laminations riemanniennes sont vues comme des généralisations des variétés riemanniennes. Nous ne faisons aucune hypothèse sur la dynamique transverse.
\end{rem}

Par compacité de l'espace ambiant, les variétés riemanniennes $(L,g_L)$ sont complètes et à géométrie uniformément bornée (courbures sectionnelles uniformément bornées et rayons d'injectivité uniformément minorés). Une vérification minutieuse de ces faits se trouve dans un article de Lessa \cite{Lessa}.

\begin{rem}
\label{smoothcurvature}
Lorsque $\dim\cL=2$, la fonction $\kappa:X\to\R$ associant à tout  $x\in X$ la courbure gaussienne en $x$ de la surface riemannienne $(L_x,g_{L_x})$ appartient à $C^{\infty}_l(X)$. En particulier ses dérivées dans les feuilles sont uniformément bornées à tout ordre.
\end{rem}

\paragraph{\textbf{Structures conformes}}  Nous disons que deux métriques feuilletées $g$ et $g'$ sont \emph{conformément équivalentes} s'il existe $u\in\Cil(X)$ telle que dans toute feuille $L$, l'égalité suivante ait lieu:
$$g_L'=e^{2u} g_L.$$
Les classes d'équivalence de cette relation s'appellent \emph{structures conformes} de la lamination $\cL$.

\subsection{Espace des revêtements riemanniens des feuilles}\label{ss.compacite}

Soit $(X,\cL,g)$ une lamination riemannienne compacte de dimension $d$. Nous considérons une famille $\cRL$ de variétés riemanniennes pointées et peintes dans $X$ de dimension $d$. Elle est formée des quadruplets $(\overline{L},\bar{g}_L,\bar{p},\bar{e})$, tel qu'il existe un revêtement riemannien sur une certaine feuille $\pi:(\overline{L},\bar{g}_L,\bar p)\to(L,g,p)$ et où $\bar e=e\circ \pi:\overline{L}\to X$ ($e$ étant le plongement canonique de $L$ dans $X$). Le but de ce paragraphe est de donner une preuve du résultat suivant.

\begin{theorem}
\label{compactness_riemannian_covers}
Soit $(X,\cL,g)$ une lamination riemannienne compacte. Alors la famille $\cRL$ est compacte pour la topologie de Cheeger-Gromov. Plus précisément soit $(\overline{L}_n,\overline{g}_{L_n},\bar{p}_n,\bar e_n)$ une suite d'éléments de $\cRL$. Alors cette suite possède une sous-suite convergente au sens de Cheeger-Gromov, et tout point d'accumulation est un quadruplet $(\overline{L},\bar{g}_L,\bar{p},\bar e)\in\cRL$.
\end{theorem}



Le Théorème \ref{compactness_riemannian_covers} est une version faible d'un théorème dû à Lessa (cf \cite[Theorem 2.3.]{Lessa}). La preuve que nous présentons ci-dessous simplifie certains de ses arguments. Nous commençons par une version laminée du théorème d'Arzela-Ascoli.

\begin{theorem}[Arzela-Ascoli laminé]
\label{Ascoli}
Soit $(X,\cL,g)$ une lamination riemannienne compacte et $M$ une variété riemannienne. Soit $\Phi_n:M\to X$  une suite d'applications envoyant $M$ dans les feuilles de $\cL$ de façon lisse. Supposons que pour tout compact $K\dans M$ et tout entier $k\in\N$, il existe $A_k>0$ tel que pour tout $n$
$$\left|\left|D^k\Phi_n|_K\right|\right|\leq A_k.$$
Il existe alors une application lisse $\Phi:M\to X$ envoyant $M$ dans une feuille de $\cL$ vers laquelle converge une sous-suite de $(\Phi_{n})$ dans la topologie $\Ciloc$.
\end{theorem}

\begin{proof}[Preuve]
Prenons un atlas fini $\cA=(U_i,T_i,\phi_i)_{i\in I}$ de cartes laminées. Nous demandons que $\cA$ soit un \emph{bon atlas laminé} au sens où les applications $\phi_i$ s'étendent en des cartes laminées, encore notées $\phi_i:V_i\to(-2,2)^d\times S_i$, où $V_i\supset U_i$ et $S_i$ contient un compact contenant $T_i$. L'ouvert $V_i$ vient avec un système continu de coordonnées transverses $\gamma_i:V_i\to S_i$ obtenu par composition de $\phi_i$ avec la projection sur la seconde coordonnée. Ainsi les $\gamma_i^{-1}(s)$, $s\in S_i$ sont précisément les plaques de $V_i$. Prenons $\epsilon>0$ tel que pour tout $i\in I$ et $x\in U_i$, la boule $B_{\cL}(x,\epsilon)$ centrée en $x$ et de rayon $\epsilon$ à l'intérieur de la feuille $L_x$, soit incluse dans $V_i$.

Soit $Y\dans M$ un ensemble dénombrable et dense. Par compacité de $X$, un argument diagonal nous permet de nous restreindre au cas où la restriction de $\Phi_n$ à $Y$ converge ponctuellement vers une application $\Phi:Y\to X$. Soit $(K_m)_{m\in\N}$ une exhaustion compacte de $M$ telle que pour tout $m$, $Y\cap K_m$ est dense dans $K_m$. Fixons $m\in\N$. Par hypothèse il existe $r>0$ tel que pour tous $x\in K_m$ et $n\in\N$, $\Phi_n(B_M(x,r))\dans B_{\cL}(\Phi_n(x),\epsilon)$. Notons que $K_m$ est recouvert par les boules $B_M(x,r)$ lorsque $x$ varie dans $Y\cap K_m$.

Choisissons à présent un point $y\in Y\cap K_m$ et un indice $i\in I$ tels que $\Phi(y)\in U_i$. En particulier il existe $n_0\in\N$ tel que pour tout $n\geq n_0$, $\Phi_n(y)\in U_i$. Ainsi, $\Phi_n(B_M(y,r))\dans B_{\cL}(\Phi_n(y),\epsilon)\dans V_i$.

Définissons une distance $d_i$ sur $V_i$ par la formule
$$d_i(\phi_i^{-1}(x,t),\phi^{-1}_i(y,s))=|x-y|+d_{S_i}(s,t),$$
où $d_{S_i}$ est n'importe quelle distance compatible avec la topologie de $S_i$. Par continuité transverse de la métrique feuilletée il existe une constante $\lambda>0$ telle que pour tous points $x,y\in V_i$ appartenant à la même plaque,
$$d_i(x,y)\leq \lambda \dist_{\cL}(x,y),$$
où $\dist_{\cL}$ représente la distance feuilletée.

Pour cette distance, la suite d'applications $(\Phi_n)_{n\geq n_0}$, restreintes à $B_M(y,r)$, et prenant valeurs dans le compact $\overline{V_i}$, est équicontinue. Il en est de même pour les suites formées par les dérivées d'ordre supérieur. Le théorème d'Arzela-Ascoli entraîne l'existence d'une sous-suite de $(\Phi_n)_{n\geq n_0}$ qui converge en restriction à $B_M(y,r)$, dans la topologie $\Ciloc$,   vers une fonction qui étend $\Phi$ et qu'on notera encore $\Phi$. De plus, puisque par hypothèse, la coordonnée transverse $\gamma_i$ est constante sur chaque $\Phi_n(B_M(y,r))$, elle doit également l'être sur $\Phi(B_M(y,r))$. En d'autres termes, $\Phi$ envoie les éléments de $B_M(y,r)$ dans une même plaque de $V_i$.

En appliquant un dernier argument diagonal, nous concluons que $(\Phi_n)$ possède une sous-suite qui converge dans la topologie $\Ciloc$ vers une application lisse $\Phi$ qui envoie $M$ tout entier dans une feuille de $\cL$.
\end{proof}

\begin{rem}
Lorsque $(X,\cL)$ est un feuilletage lisse d'une variété lisse, et que $g$ est une métrique riemannienne sur $X$ induisant une métrique dans les feuilles, la preuve est bien plus simple. En effet, dans ce cas la distance riemannienne $d$ sur $X$ vérifie
\begin{equation}\label{adaptee}
d(x,y)\leq\dist_{\cL}(x,y)
\end{equation}
pour toute paire de points $(x,y)$ appartenant à une même feuille. Le lemme précédent est alors une application immédiate du théorème d'Arzela-Ascoli.
\end{rem}

\begin{rem}
L'approche de Lessa pour prouver le Théorème \ref{Ascoli} pour une lamination quelconque consiste à démontrer l'existence d'une distance globale, qu'il appelle \emph{adaptée}, satisfaisant à \eqref{adaptee} pour toute paire de points appartenant à une même feuille. C'est un des principaux points techniques de son travail (cf. \cite[Lemma 9.1]{Lessa}). Nous pensons que l'approche proposée ici est plus simple.

\end{rem}

\begin{rem}\label{rem_fonctions_CG}
Si $u\in C^\infty_l(X)$ alors la suite de fonctions lisses sur $M$ données par $u\circ \Phi_n$ possède une sous-suite convergeant vers $u\circ\Phi$ au sens $\Ciloc$.
\end{rem}

\begin{proof}[Preuve du Théorème \ref{compactness_riemannian_covers}] Soit $(X,\cL,g)$ une lamination riemannienne compacte. Prenons une suite $(\overline{L}_n,\overline{g}_{L_n},\bar p_n,\bar e_n)$ d'éléments de $\cRL$. La suite de variétés pointées $(\overline{L}_n,\overline{g}_{L_n},\bar p_n)$ est à géométrie uniformément bornée. Elle possède par conséquent des sous-suites convergentes: nous renvoyons à \cite{Ch,Gr,Lessa,Pet} pour ce fait. Supposons qu'elle converge vers une variété riemannienne pointée $(\overline{L},\overline{g},\bar p)$. Soit $F_n:\overline{L}\to\overline{L}_n$ une suite d'applications de convergence.

Par définition des applications de convergence, les applications $\bar e_n\circ F_n$ ont toutes leurs dérivées uniformément bornées sur les compacts. Le théorème d'Arzela-Ascoli laminé (Théorème \ref{Ascoli}) entraîne qu'elle possède une sous-suite convergente dans la topologie $\Ciloc$ vers une application lisse $\bar e:\overline{L}\to X$, qui envoie $\overline{L}$ dans une feuille $L$. Par définition de $\bar e_n$ et $F_n$ cette application est une isométrie locale sur son image. Ainsi, puisque $L$ est connexe et que $\overline{L}$ est connexe et complète, c'est un revêtement riemannien sur son image. Ainsi $(\overline{L},\bar g,\bar p,\bar e)$ est un élément de $\cRL$, ce qui achève la preuve du théorème.
\end{proof}

\subsection{Espace des revêtements universels des feuilles}\label{ss.rev_univ}

Soit $\cUL$ la sous-famille de $\cRL$ formée par les revêtements universels de feuilles de $\cL$. Rappelons que la limite de Cheeger-Gromov d'une suite de variétés simplement connexes n'est pas nécessairement simplement connexe donc $\cUL$ pourrait ne pas être fermé dans $\cRL$. Nous donnons une condition topologique sur la lamination assurant que cette famille soit fermée. Ce résultat de fermeture est un ingrédient essentiel de la preuve du Théorème \ref{prescription_courboubouille}.

\begin{theorem}
\label{compactness_univ-cover}
Soit $(X,\cL,g)$ une lamination riemannienne compacte sans cycle évanouissant.  Alors la famille $\cUL$ est compacte pour la topologie de Cheeger-Gromov.
\end{theorem}

Ce théorème est à comparer avec le Corollaire \ref{compact_univ_covers}. Avant de prouver le théorème, définissons les cycles évanouissants. 

\begin{defn}
\label{defvancycle}
Soit $(X,\cL)$ une lamination compacte. Un \emph{cycle évanouissant} est un lacet $c$ inclus  \emph{et essentiel} dans une feuille $L$, qui est accumulé par une suite de lacets $c_n$ inclus \emph{et triviaux} dans des feuilles $L_n$.
\end{defn}
%

\paragraph{\textbf{Cycles évanouissants et feuilletages de $3$-variétés}}
L'existence d'un cycle évanouissant pour un \emph{feuilletage lisse} $\cF$ de dimension $2$ d'une  \emph{variété} $M$ de dimension $3$ impose de sévères restrictions topologiques. Si une feuille $L$ d'un tel feuilletage supporte un cycle évanouissant elle doit être un tore bordant un tore solide dans lequel $\cF$ induit le feuilletage de Reeb (c'est une \emph{composante de Reeb}): c'est le théorème de Novikov (voir par exemple \cite[Chapter 7]{CLN}). Nous renvoyons également au travail d'Alcalde Cuesta et Hector \cite{Alcalde_Hector}, pour un critère de trivialité pour les cycles évanouissants pour les feuilletages par surfaces en codimension quelconque.

Nous rappelons que le feuilletage de Reeb est obtenu par le processus suivant. Partons du feuilletage du cylindre plein $\overline{\D}\times\R$ dont les feuilles sont les surfaces définies par $(x,\frac{|x^2|}{1-|x^2|}+a)$ où $a$ est un paramètre réel et $|x|<1$, et le bord. Ce feuilletage est invariant par translation $(x,t)\mapsto(x,t+1)$ et induit ce qu'on appelle le feuilletage de Reeb dans le tore solide $\overline{\D}\times\Ss^1$. Il a une feuille torique (son bord), et les autres feuilles sont toutes des plans s'accumulant au bord.

\begin{rem}
Il est nécessaire d'exclure l'existence de cycles évanouissants pour garantir la compacité de l'espace des revêtements universels des feuilles. En effet ce théorème ne s'applique pas au feuilletage de Reeb introduit ci-dessus. Pour toute suite de points $(p_n)$ intérieurs au tore solide convergeant vers un point du bord, la suite de feuilles $(L_{p_n},p_n)$, qui sont des plans, converge vers un cylindre (le revêtement d'holonomie du bord).
\end{rem}

\begin{proof}[Preuve du Théorème \ref{compactness_univ-cover}] D'après le Théorème \ref{compactness_riemannian_covers} il ne nous reste qu'à prouver que si une suite de revêtements universels pointés de feuilles converge au sens de Cheeger-Gromov, alors la limite est simplement connexe. 

Pour ce faire, prenons une suite $(\tL_n,\widetilde{g}_{L_n},\tilde p_n,\tilde e_n)$ d'éléments de $\cUL$ convergeant vers $(\overline{L},\overline{g}_L,\bar p,\bar e)$ au sens de Cheeger-Gromov et soit $F_n:\overline{L}\to\tL_n$ une suite d'applications de convergence. Supposons que $\overline{L}$ ne soit pas simplement connexe. Alors cette variété contient un lacet \emph{essentiel} $\bar c$. Sa projection sur $L$, notée $c=\bar e(\bar c)$, l'est également. De plus, lorsque $n$ est suffisamment grand, $c_n=\tilde e_n\circ F_n(\bar c)$ est une suite de lacets \emph{triviaux} inclus dans $L_n$ et qui converge uniformément vers $c$. Ceci implique que $c$ est un cycle évanouissant, contredisant l'hypothèse.
\end{proof}

Le résultat suivant est bien connu, et découle directement de la stabilité des géodésiques fermées en courbure négative (cf. \cite{BGS}). Il permet d'appliquer le Théorème \ref{compactness_univ-cover} pour les laminations dont les feuilles sont à courbure sectionnelle négative.

\begin{proposition}\label{no_vn_cycles} Soit $(X,\cL)$ une lamination riemannienne compacte dont les feuilles sont à courbure sectionnelle négative. Alors $(X,\cL)$ n'a pas de cycle évanouissant. En particulier l'espace $\cUL$ des revêtements universels pointés de feuilles de $\cL$ est compact pour la topologie de Cheeger-Gromov.
\end{proposition}

\begin{rem}\label{rem_groupoide}
Le rapporteur nous a fait remarquer l'existence d'un plongement naturel $X\to\cUL$ dont la construction utilise le \emph{groupoïde d'homotopie} $\Pi_1(\cL)$ de la lamination. Rappelons que cet objet est obtenu comme le quotient de l'espace des chemins tangents à la lamination quotienté par la relation d'homotopie et qu'il vient avec deux projections $\alpha$ et $\beta$ sur $X$ associant respectivement à chaque classe d'homotopie $[\gamma]$ les points initial et final de n'importe quel représentant $\gamma$ (voir  \cite{Alcalde_groupoides} pour plus de détails). Ainsi, étant donné $p\in X$ et $L=L_p$, l'application $\beta:\alpha^{-1}(p)\to \beta(\alpha^{-1}(p))$ est une application de revêtement universel de $L$, qu'on peut pointer en la classe de lacets triviaux, et qui est naturellement peinte dans $X$. Ceci fournit donc une injection $X\to\cUL$ qui est continue dès lors que $\cL$ n'a pas de cycle évanouissant (par le Théorème \ref{compactness_univ-cover}). Il serait intéressant d'étudier cette application plus en détail.
\end{rem}

\subsection{Caractérisation des familles de métriques laminées}\label{ss.carac_metrique_laminee} Une famille $(h_L)_{L\in\cL}$  de métriques riemanniennes complètes dans les feuilles d'une lamination riemannienne compacte fournit une famille de métriques sur les éléments $(\overline{L},\bar g,\bar p, \bar e)$ de l'espace $\cRL$ en posant $\bar h_L=\bar e^\ast h_L$.

\begin{lemma}[Caractérisation des métriques $C^\infty_l$]\label{l.carac_Cil}
Soit $(X,\cL,g)$ une lamination riemannienne compacte et $h=(h_L)_{L\in\cL}$ une famille de métriques riemanniennes complètes définies dans les feuilles de $X$. Alors $h$ est de type $C^\infty_l$ si et seulement si elle induit une famille de métriques $(\bar h_L)_{L\in\cL}$ dans $\cRL$ qui soit continue  au sens de Cheeger-Gromov (voir le \S \ref{sss.continuite_metriques}). Lorsque, de plus, $\cUL$ est fermé dans $\cRL$, il suffit que $h$ induise une famille de métriques continues dans $\cUL$.
\end{lemma}

\begin{proof}[Preuve] Supposons que $h$ induise une famille de métriques dans $\cRL$  continue au sens du \S \ref{sss.continuite_metriques}. Soit $(U,T,\phi)$ une carte laminée et $\psi:(-1,1)^d\times T\to U$ l'inverse de $\phi$. Soit $(x_n,t_n)\in(-1,1)^d \times T$ une suite convergeant vers $(x,t)\in(-1,1)^d\times T$. Notons $\pi_1:(-1,1)^d\times T\to(-1,1)^d$ la projection sur la première coordonnée et considérons les applications $\iota_n,\iota:(-1,1)^d\to(-1,1)^d\times T$ définies par $\iota_n(x)=(x,t_n)$ et $\iota(x)=(x,t)$. Enfin, notons $p_n=\psi(x_n,t_n)$, $L_n=L_{p_n}$, $g_n=g_{L_{p_n}}$, $h_n=h_{L_{p_n}}$ ainsi que $p=\psi(x,t)$, $L=L_{p}$, $g=g_{L_{p}}$ et $h=h_{L_p}$. Il suffit alors de prouver que
$$\lim_{n\to\infty}\iota_n^\ast\psi^\ast h_n=\iota^\ast\psi^\ast h,$$
où la limite est prise au sens $\Ciloc$. Par le Théorème \ref{compactness_riemannian_covers}, quitte à extraire, nous pouvons supposer que la famille $(L_n,g_n,p_n,e_n)$ converge au sens de Cheeger-Gromov vers $(\overline{L},\bar g,\bar p,\bar e)$: nous notons $\bar h=\bar e^\ast h$. Par continuité de la famille de métriques au sens du \S \ref{sss.continuite_metriques}, il existe une suite d'applications de convergence $(F_n)$  telle que $F_n^\ast h_n$ converge vers $\bar h$ et $e_n\circ F_n$ converge vers $\bar e$ au sens $\Ciloc$ (voir la preuve du Théorème \ref{compactness_riemannian_covers}).

Soit $V$ un voisinage de $\bar p$ dans $\overline{L}$ que $\bar e$ envoie difféomorphiquement sur son image. Supposons en outre que $\bar e(V)\dans U$. Ainsi l'application $\pi_1\circ\phi\circ\bar e$ envoie $V$ difféomorphiquement sur son image. Notons $\tau$ l'application inverse. Soit $W\dans (-1,1)^d$ un voisinage de $x$ dont l'adhérence est une partie compacte de l'image de $V$. Notons que $(\psi\circ\iota)\circ(\pi_1\circ\phi\circ\bar e)=\bar e$, d'où $\psi\circ\iota=\bar e\circ\tau.$ Ainsi
$$\iota^\ast\psi^\ast h=\tau^\ast\bar e^\ast h=\tau^\ast\bar h.$$

Considérons à présent un entier $n_0$ tel que pour tout $n\geq n_0$, $\pi_1\circ\phi\circ e_n\circ F_n$ envoie difféomorphiquement un voisinage de $p_n$ sur un ouvert contenant $\overline{W}$, et notons $\tau_n$ l'inverse de cette fonction. Nous obtenons de même que précédemment 
$$\iota_n^\ast\psi^\ast h_n=\tau_n^\ast F_n^\ast h_n. $$
Par continuité de la famille $h$ nous obtenons
$$\lim_{n\to\infty}\iota_n^\ast\psi^\ast h_n=\lim_{n\to\infty}\tau_n^\ast F_n^\ast  h_n = \tau^\ast e^\ast \bar h=\iota^\ast\psi^\ast h,$$
la convergence ayant lieu dans la topologie $\Ciloc$ sur l'ouvert $W\dans(-1,1)^d$. Cela prouve que la famille de métriques $(h_L)_{L\in\cL}$ est de type $C^\infty_l$. La réciproque est une conséquence immédiate de notre théorème d'Arzela-Ascoli laminé (voir le Théorème \ref{Ascoli}).
\end{proof}

\section{Prescription de courbure et uniformisation simultanée}
\label{simunif}

\subsection{Laminations par surfaces hyperboliques}

\paragraph{\textbf{Laminations par surfaces de Riemann}} Soit $(X,\cL)$ une lamination compacte de dimension $2$. Lorsque les applications de transition ``tangentielles'' $\zeta_{ij}(.,t)$ sont holomorphes, les feuilles sont naturellement munies de structures de surfaces de Riemann. Nous disons alors que $\cL$ est une \emph{lamination par surfaces de Riemann}.

Nous disons que les \emph{structures complexes} définies par deux laminations par surfaces de Riemann $\cL$ et $\cL'$ sur  $X$ sont \emph{équivalentes} s'il existe un homéomorphisme $h:X\to X$ envoyant feuilles de $\cL$ sur feuilles de $\cL'$ biholomorphiquement.

Les structures conformes et complexes sur une même lamination sont en correspondance bijective. Plus précisément, en utilisant la théorie d'Ahlfors-Bers (voir \cite{AhBe}), il est prouvé dans \cite[p. 232]{MS} (voir également \cite[Theorem 3.2]{Candel}) que les propriétés suivantes sont vérifiées par toute lamination par surfaces lisses $(X,\cL)$.

\begin{enumerate}
\item Pour toute métrique laminée $g$ sur $X$ il existe un atlas laminé $\cA=(U_i,T_i,\phi_i)_{i\in I}$ par \emph{cartes isothermes laminées}; c'est-à-dire que lorsqu'on la lit dans le système de coordonnées donné par $\phi_i$, la métrique laminée est conformément euclidienne dans les plaques.
\item Un tel atlas munit $\cL$ d'une structure complexe; les structures complexes associées à deux métriques laminées $g$ et $g'$ sont équivalentes si et seulement si les métriques sont conformément équivalentes.
\item Toute structure complexe est induite par une métrique laminée de cette façon.
\end{enumerate}

\paragraph{\textbf{Laminations par surfaces hyperboliques}} Une feuille d'une lamination par surfaces de Riemann peut être uniformisée soit par la sphère $\C\PP^1$, soit par le plan $\C$, soit par le disque $\D$. Nous disons que $(X,\cL)$ est une \emph{lamination par surfaces hyperboliques } si ses feuilles sont toutes uniformisées par le disque.

Candel prouve dans \cite[Theorem 4.3]{Candel} que le fait d'être une lamination par surfaces hyperboliques est \emph{topologique}, et indépendant de la structure complexe à proprement parler. En particulier, c'est le cas de toute lamination par surfaces sans mesure transverse invariante par holonomie. De nombreux exemples se trouvent dans \cite{Ghys_Laminations}. Nous renvoyons également à \cite{AB,ABMP} pour une étude topologique des surfaces pouvant apparaître comme feuilles de telles laminations.

\subsection{Uniformisation simultanée}\label{ss.uniformouille} Nous sommes désormais prêts à donner une preuve alternative du théorème d'uniformisation simultanée de Candel.

\begin{theorem}[Candel]\label{Candel}
Soit $(X,\cL)$ une lamination compacte par surfaces hyperboliques. Alors il existe, dans la classe conforme correspondante, une unique métrique laminée de courbure gaussienne constante égale à $-1$.
\end{theorem}
\begin{proof}[Preuve] Munissons $(X,\cL)$ d'une métrique laminée $g$ dans la classe conforme correspondante. Le Théorème \ref{compact_unif} fournit la continuité des transformations d'uniformisation, et le Théorème  \ref{compactness_riemannian_covers}, la compacité de l'espace $\cRL$. Mieux: en utilisant le Corollaire \ref{compact_univ_covers} nous obtenons la compacité de l'espace $\cUL$ (qui est alors fermé dans $\cRL$). Soit alors $(\widetilde{L}_n,\tilde g_n,\tilde p_n,\tilde e_n)$ une suite d'éléments de $\cUL$ convergeant vers $(\widetilde{L},\tilde g,\tilde p,\tilde e)\in\cUL$, $(F_n)$ une suite d'applications de convergence correspondante, $\Pi_n:(\D,0)\to(\widetilde{L}_n,\tilde p_n)$ et $\Pi:(\D,0)\to(\widetilde{L},\tilde p)$ les applications d'uniformisation, et $\tilde h_n=\Pi_n\,_\ast h_\D$ et $\tilde h=\Pi_\ast h_\D$ les métriques hyperboliques dans la classe conforme. Nous avons alors $\tilde e_n\circ F_n\to\tilde e$ et $F_n^\ast\tilde h_n=F_n^\ast(\Pi_n\,_\ast h_\D)\to \Pi_\ast h_\D=\tilde h$  au sens $\Ciloc$ (par continuité des applications d'uniformisation). Ceci prouve la continuité de la famille $(\tilde h_L)_{L\in\cL}$. On en déduit, par le Lemme \ref{l.carac_Cil}, que la famille des métriques hyperboliques le long des feuilles de $\cL$ appartient à $C^\infty_l$.
\end{proof}

\subsection{Métriques laminées à courbure prescrite}\label{ss.met_lam_courb_pres}

Nous renforçons à présent le Théorème \ref{Candel} en prouvant le Théorème \ref{prescription_courboubouille}. La première étape dans la preuve de ce théorème est de démontrer l'existence d'une métrique laminée à courbure négative dans les feuilles. Le Théorème \ref{Candel} nous fournit une telle métrique. Nous pouvons également obtenir une telle métrique en utilisant un Théorème dû à Ghys, dont la preuve apparaît dans \cite{AY}, et que nous reproduisons en appendice pour la commodité du lecteur. Nous obtenons ainsi une autre preuve du Théorème \ref{Candel}.

\paragraph{\textbf{Espaces de H\"older}}  Nous commencerons par rappeler quelques concepts élémentaires de la théorie des espaces de H\"older (nous renvoyons à \cite{GT} pour plus d'informations).

Soit $(M,g)$ une variété riemannienne et $\alpha\in[0,1]$. La \emph{semi-norme de H\"older d'ordre $\alpha$} d'une fonction continue $f:M\to\R$ est définie par
$$[f]_\alpha=\sup_{x\neq y} \frac{|f(x)-f(y)|}{d(x,y)^\alpha},$$
où $d(x,y)$ représente la distance riemannienne par rapport à $g$. Les semi-normes de sections de fibrés sur $M$ se définissent de manière analogue. Pour tout $k\in\N$ et $\alpha\in[0,1]$, la \emph{norme de H\"older d'ordre $(k,\alpha)$} d'une fonction $k$ fois différentiable se définit par
$$\|f\|_{C^{k,\alpha}}=\sum_{i=0}^k\|\nabla^i f\|_{C^0}+\left[\nabla^k f\right]_\alpha,$$
$\nabla$ représentant la dérivée covariante de Levi-Civita sur $M$. Nous définissons alors l'\emph{espace de H\"older d'ordre $(k,\alpha)$} comme
$$C^{k,\alpha}(M)=\left\{f\in C^k(M);\,\|f\|_{C^{k,\alpha}}<\infty\right\}.$$

\begin{rem}\label{rem.Banach}
L'espace $C^{k,\alpha}(M)$, muni de la norme $||.||_{C^{k,\alpha}}$, est un espace de Banach.
\end{rem}

\begin{rem}\label{rem.normes_equ}
Notons que la norme de Hölder d'ordre $(k,\alpha)$ est équivalente à la norme donnée par
$$\|f\|_{C^{k,\alpha}}'=\sup_{x\in M}\|f\|_{C^{k,\alpha}(B(x,1))}.$$
\end{rem}

Finalement, nous définissons
$$C^{\infty}_b(M)=\bigcap_{k+\alpha\geq 0}C^{k+\alpha}(M).$$
Remarquons que cet espace diffère de l'espace $C^\infty(M)$ en ce que des éléments de ce dernier espace pourraient être arbitrairement grands près de l'infini.

\paragraph{\textbf{Prescription de courbure}} Afin de prouver le Théorème \ref{prescription_courboubouille} nous allons énoncer un théorème d'Aviles-McOwen et Bland-Kalka (voir \cite{AMc,BK}) sous une forme un peu plus générale. Nous en donnerons une preuve directe dans la Section \ref{s.presctiption} car nous aurons besoin de bornes uniformes des normes de H\"older pour prouver le Théorème \ref{prescription_courboubouille}, ainsi que de la variation lisse des solutions.

\begin{theorem}\label{th_bounds_elliptic}
Soit $(M, g)$ une surface riemannienne complète et simplement connexe de courbure gaussienne $-\kappa_0$ uniformément pincée entre deux constantes strictement négatives, et telle que $\log(\kappa_0)\in C^\infty_b(M)$.  Supposons que $\kappa : M\to(0,\infty)$ soit telle que $\log(\kappa)\in C^\infty_b(M)$. Alors il existe une unique fonction lisse  $u : M\to\R$
telle que
\begin{enumerate}
\item  la métrique $e^{2u}g$ est complète; et
\item sa courbure gaussienne est donnée en tout point par $-\kappa$.
\end{enumerate}
De plus, pour tout $(k,\alpha)$, il existe une constante positive $C_{k,\alpha}$ ne dépendant que de $||\log(\kappa)||_{C^{k,\alpha}}$ et des bornes sur la géométrie de $M$ telle que
$$\|u\|_{C^{k+2,\alpha}}\leq C_{k,\alpha}.$$
\end{theorem}

\begin{proof}[Preuve du Théorème \ref{prescription_courboubouille}]
Soit $\kappa:X\to(0,\infty)$ une fonction vérifiant les hypothèses du théorème et $g$ une métrique laminée dont la courbure est négative partout. Elle peut venir soit du théorème d'uniformisation simultanée de Candel, soit de l'argument de Ghys présenté au \S \ref{ss.gaussboubouille}, si l'on veut avoir une preuve indépendante de l'uniformisation simultanée (ainsi qu'une nouvelle preuve de ce résultat).

Nous considérons l'espace des revêtements universels des feuilles $\cUL$ formé par les quadruplets $(\widetilde{L},\tilde g_L,\tilde p,\tilde e)$ introduits au \S \ref{ss.rev_univ}. La courbure des feuilles étant négative, cet espace est compact (voir le Corollaire \ref{compact_univ_covers} ou le Théorème \ref{compactness_univ-cover}). La fonction $\kappa$ fournit une famille de fonctions $\tilde \kappa_L$ strictement positives définies sur les éléments de $\cUL$ dont les logarithmes ainsi que les dérivées à tout ordre sont uniformément bornés. Par le Théorème \ref{th_bounds_elliptic} il existe une unique famille de fonctions lisses $\tilde u_L$ sur $\widetilde{L}$ telles que 
\begin{enumerate}
\item la métrique $\tilde{h}_{L}=e^{2\tilde{u}_{L}}\tilde{g}_{L}$ est complète et de courbure en tout point donnée par $-\widetilde{\kappa}_L$; et
\item pour tout $(k,\alpha)$,  $\|\tilde{u}_{L}\|_{C^{k+2,\alpha}}\leq C_{k,\alpha}$,
\end{enumerate}
pour une famille de constantes $C_{k,\alpha}$ ne dépendant que de $\kappa$ et des bornes de la géométrie de la métrique laminée $g$. Il découle de l'unicité de la métrique et de l'invariance des objets $\tilde g_L$, $\tilde \kappa_L$ par automorphismes de revêtement, que la famille de métriques $(\tilde h_L)$ descend en une famille de métriques dans les feuilles de $\cL$ dont la courbure est en tout point donnée par $-\kappa$ (ce qui les détermine uniquement). Par le Lemme \ref{l.carac_Cil}, afin de démontrer que cette métrique laminée est de classe $C^\infty_l$ sur $X$ il suffit de prouver qu'elle est continue au sens de Cheeger-Gromov (voir \S \ref{sss.continuite_metriques}).

Considérons donc une famille $(\widetilde{L}_n,\tilde g_n,\tilde p_n,\tilde e_n)$ d'éléments de $\cUL$ convergeant vers $(\widetilde{L},\tilde g,\tilde p,\tilde e)$ ainsi que les objets correspondants $\tilde \kappa_n,\tilde \kappa, \tilde u_n,\tilde u,\tilde h_n,\tilde h$, et qu'une suite d'applications de convergence $(F_n)$ vérifiant de plus que $\tilde e_n\circ F_n\to\tilde e$ au sens $\Ciloc$. Pour tout compact $K\dans\widetilde{L}$, il existe une famille de constantes $C'_{k,\alpha,K}$, indépendantes de $n$, telles que pour tous $n,k,\alpha$, $||\tilde u_n\circ F_n||_{C^{k+2,\alpha}(K)}\leq C'_{k,\alpha,K}$. En utilisant le théorème d'Arzela-Ascoli, ainsi qu'un argument diagonal, nous pouvons supposer que $\tilde u_n\circ F_n$ converge vers une fonction lisse $v$ sur $\widetilde{L}$ dans la topologie $\Ciloc$. Nous déduisons deux choses de l'égalité $F_n^\ast\tilde h_n=e^{2\tilde{u}_{n}\circ F_n}\,F_n\,^\ast\tilde{g}_{n}$. Premièrement la courbure de cette métrique vaut $\widetilde{\kappa}_n\circ F_n$ en tout point de $\widetilde{L}$. Deuxièmement cette métrique converge vers $e^{2v}\tilde{g}_L$, qui est complète, dans la topologie $\Ciloc$. La fonction $\kappa:X\to(0,\infty)$ étant de classe $C^{\infty}_l$, nous déduisons de la Remarque \ref{rem_fonctions_CG} que la suite $\tilde \kappa_n\circ F_n$ converge vers $\tilde \kappa$ dans la topologie $\Ciloc$. La métrique $e^{2v}\tilde g$ a donc courbure $\tilde \kappa$ et doit être égale à $\tilde{h}$, par unicité. Nous en déduisons que $F_n^\ast\tilde h_n$ converge vers $\tilde h$ dans la topologie $\Ciloc$. Ceci conclut la preuve du théorème.
\end{proof}

\section{Prescription de courbure de disques riemanniens et variation lisse}\label{s.presctiption}

\subsection{Prescription de courbure de disques riemanniens}\label{ss.prescription_disque}

Nous donnons ci-dessous une preuve auto-contenue du Théorème \ref{th_bounds_elliptic}. Nous commencerons par prouver un principe du maximum à l'infini.

\begin{lemma}\label{l.inequality_laplace_phi}
Soit $N$ une variété riemannienne complète de dimension $d$, de courbure sectionnelle et rayon d'injectivité respectivement minorés par $-C^2$ et $1/C$ pour une certaine constante $C>0$. Soit $\Phi:\R\to\R$ une fonction telle que
\begin{enumerate}
\item $\Phi$ est croissante sur $[C,\infty)$; et
\item $\lim_{t\to\infty}\Phi(t)t^{-(1+1/C)}=\infty.$
\end{enumerate}
Si $u:N\to\R$ est une fonction deux fois différentiable telle que
$$\Delta u\geq\Phi\circ u,$$
alors il existe $A>0$ ne dépendant que de $C$ et $\Phi$ tels que
$$\sup_{x\in N} u(x)\leq A.$$
\end{lemma}

\begin{proof}[Preuve]
Prenons un point $x_0\in N$ vérifiant $u(x_0)\geq\max(C,1)$. Par complétude, nous pouvons utiliser le Lemme du quasi-maximum \ref{quasi-max} et supposer que $u(x)\leq 2u(x_0)$ pour tout $x\in B$ où
$$B=B\left(x_0,u(x_0)^{-1/(2C)}/C\right).$$
Soit la fonction lisse $v:B\to\R$ donnée par
$$v(x)=d(x,x_0)^2.$$
Par comparaison nous avons, pour tout $x\neq x_0$ dans $\overline{B}$,
$$\Delta v(x)\leq 2(d-1)C d(x,x_0)\coth\left(Cd(x,x_0)\right)\leq C_1,$$
où $C_1$ ne dépend que de $C$. Soit maintenant la fonction $w:B\to\R$
donnée par
$$w(x):=u(x)-2C^2u(x_0)^{1+1/C}v(x).$$
Alors $w\leq 0$ sur $\partial B$ et $w(x_0)=u(x_0)\geq 1$ ainsi $w$ atteint son maximum en un point
intérieur $x$ de $\overline{B}$. Par le principe du maximum, nous avons en ce point, $\Delta w(x)\leq 0$ d'où
$$\Delta u(x)\leq 2C^2 C_1 u(x_0)^{1+1/C}.$$
Mais par hypothèse,
$$\Delta u(x)\geq\Phi(u(x))\geq\Phi(u(x_0)),$$
d'où
$$\Phi(u(x_0))u(x_0)^{-(1+1/C)}\leq 2C^2C_1.$$
Le résultat suit en utilisant la seconde propriété de l'hypothèse.
\end{proof}

\begin{rem}
Nous renvoyons à \cite{Osserman} pour une étude détaillée de l'inéquation $\Delta u\geq \Phi\circ u$.
\end{rem}

\begin{lemma}\label{l.invertouille}
Soit $M$ un disque riemannien complet de courbure partout minorée par $-C^2$, pour un certain $C>0$. Soit $\phi\in C^\infty_b(M)$ telle que $e^\phi\geq\eta$ pour une certaine constante $\eta>0$. Alors pour toute paire $(k,\alpha)$, l'opérateur $\Delta-e^\phi$ définit un isomorphisme linéaire entre $C^{k+2,\alpha}(M)$ et $C^{k,\alpha}(M)$.
\end{lemma}
\begin{proof}[Preuve] Soit $L=\Delta-e^\phi$. Puisque $\phi\in C^\infty_b(M)$ et que $M$ est à géométrie bornée, nous déduisons des estimées de Schauder classiques (pour lesquelles nous renvoyons à \cite[Corollary 6.3]{GT}) qu'il existe une constante $C_0>0$ telle pour tout $x\in M$ et pour tout $u\in C^{k+2,\alpha}(M)$
$$\|u\|_{C^{k+2,\alpha}(B(x,1))}\leq C_0\left(\|u\|_{C^0(B(x,2))}+\|Lu\|_{C^{k,\alpha}(B(x,2))}\right).$$
En utilisant la Remarque \ref{rem.normes_equ} nous trouvons $C_1>0$ tel que pour tout $u\in C^{k+2,\alpha}(M)$,
\begin{equation}\label{Schaudouille}
\|u\|_{C^{k+2,\alpha}}\leq C_1(\|u\|_{C^0}+\|Lu\|_{C^{k,\alpha}}),
\end{equation}
pour une certaine constante  $C_1>0$. Notons que, la variété $M$ n'étant pas compacte, l'inclusion $C^{k+2,\alpha}(M)\hookrightarrow C^0(M)$ n'est pas compacte et l'estimée ci-dessus n'est pas une estimée elliptique.

Nous allons raffiner le premier terme du membre de droite de \eqref{Schaudouille}. Prenons $x_0\in M$ et $r_0>0$ et notons  $B=B(x_0,,r_0)$. Pour $a, b, \lambda > 0$, nous considérons la fonction radiale
$$v(x):=a\cosh(\lambda(d(x,x_0)-b)).$$
Par hypothèse, la courbure de $M$ est minorée par $-C^2$. Donc par comparaison, lorsque $0<\lambda<C$,
$$Lv\leq v_{rr}+C\coth(Cr)v_r-e^\phi v,$$
où l'indice $r$ représente la dérivée dans la direction radiale issue de $x_0$.
En particulier si $\lambda>0$ est assez petit nous avons,
$$Lv\leq (\lambda^2+\lambda C\coth(Cr)-e^\phi)v\leq-\frac{\eta a}{2}.$$
\`A présent, considérons $u\in C^{k+2,\alpha}(M)$  et posons
$$a:=||u||_{C^0(B)}+\frac{2}{\eta}\left|\left|Lu\right|\right|_{C^0}.$$
Le principe du maximum nous donne $u-v\leq 0$ sur tout $M\setminus B$. Puisque $b\in\R$ est arbitraire, nous trouvons que sur cette région,
$$u\leq a=||u||_{C^0(B)}+\frac{2}{\eta}||Lu||_{C^0}.$$
Les bornes inférieures s'obtiennent de la même façon. Donc, en utilisant l'estimée \eqref{Schaudouille} on obtient
\begin{equation}\label{Better_Schauder}
||u||_{C^{k+2,\alpha}}\leq C_2\left(||u||_{C^0(B)}+||L u||_{C^{k,\alpha}}\right),
\end{equation}

Un raffinement de l'argument ci-dessus prouve en fait que les éléments du noyau $\ker(L)$ décroissent exponentiellement à l'infini donc, par le principe du maximum de nouveau, $\ker(L)$ est trivial. Un argument de type ``blow-up'' appliqué à l'estimée \eqref{Better_Schauder} donne maintenant
\begin{equation}\label{blow_up}
||u||_{C^{k+2,\alpha}}\leq C_3||Lu||_{C^{k,\alpha}},
\end{equation}
pour une constante $C_3$. Cela prouve que $L$ définit une application linéaire injective de $C^{k+2,\alpha}(M)$ vers un sous-espace fermé de $C^{k,\alpha}(M)$. La surjectivité découle d'un argument d'approximation par des fonctions lisses à supports compacts, ce qui complète la preuve.
\end{proof}

\begin{rem}
\label{rem.bornes_constantes}
Les constantes $C_0,C_1,C_2$ et $C_3$ obtenues dans la preuve précédente ne dépendent que des bornes sur la géométrie de $M$ et de $||\phi||_{C^{k,\alpha}}$.
\end{rem}

Dans la suite, $(M,g)$ représente un disque riemannien complet, de courbure $\kappa_0$ partout pincée entre deux constantes négatives $-C^2<-1/C^2<0$. Le prochain lemme nous permet de borner le facteur conforme, ainsi que ses dérivées, d'une métrique conforme à courbure négative prescrite.

\begin{lemma}\label{bornouille}
Soit $\kappa:M\to(0,\infty)$ une fonction telle que $\log(\kappa)\in C^\infty_b(M)$. Supposons qu'il existe une fonction lisse $u:M\to\R$ telle que la métrique $e^{2u}g$ soit complète et de courbure partout donnée par $-\kappa$. Alors $u\in C^\infty_b(M)$ et pour toute paire $(k,\alpha)$, il existe une constante $C_{k,\alpha}$ ne dépendant que de $\|\log(\kappa)\|_{C^{k,\alpha}}$ et de la géométrie de $M$ telle que
$$\|u\|_{C^{k+2,\alpha}}\leq C_{k,\alpha}.$$
\end{lemma}

\begin{proof}[Preuve]
Nous rappelons (voir \cite{Chavel,ChowKnopf}) que
\begin{equation}\label{eq.courbure_u}
\Delta u=e^{2u}\kappa-\kappa_0,\,\,\text{et}
\end{equation}
\begin{equation}\label{eq.courbure_moinsu}\widetilde{\Delta}(-u)=e^{-2u}\kappa_0-\kappa.
\end{equation}
où $\Delta$ et $\widetilde{\Delta}$ représentent les opérateurs de Laplace-Beltrami de $g$ et $\tilde{g}=e^{2u}g$ respectivement. Ces métriques étant toutes deux complètes, le Lemme \ref{l.inequality_laplace_phi} nous fournit une constante $C_0>0$ ne dépendant que de $\|\log(\kappa)\|_{C^0}$ et de $\|\log(\kappa_0)\|_{C^0}$ telle que
$$\|u\|_{C^0}\leq C_0.$$

Puisque $M$ est à géométrie bornée, et que $u$ est solution d'une EDP quasilinéaire, la théorie elliptique classique nous permet de prouver que les dérivées de $u$ sont bornées à tout ordre, donc $u\in C^\infty_b(M)$ ce qui prouve la première assertion. Reste à majorer uniformément $||u||_{C^{k,\alpha}}$ par récurrence.

Nous allons dans un premier temps borner $||u||_{C^{2,\alpha}}$. Appliquons  le Lemme \ref{l.invertouille} avec $e^\phi=\kappa$. L'opérateur $L=\Delta-\kappa:C^{2,\alpha}\to C^{0,\alpha}$ définit alors un isomorphisme linéaire entre $C^{2,\alpha}(M)$ et $C^{0,\alpha}(M)$ et il existe une constante $C_1>0$ ne dépendant que de la géométrie de $M$ et de $||\log(\kappa)||_{C^{0,\alpha}}$ (voir la Remarque \ref{rem.bornes_constantes}) telle que
$$||u||_{C^{2,\alpha}}\leq C_1||L u||_{C^{0,\alpha}}=C_1\left|\left|(e^{2u}-u)\kappa-\kappa_0\right|\right|_{C^{0,\alpha}}\leq C_2\left(1+||u||_{C^{0,\alpha}}\right),$$
où $C_2$ dépénd de $C_1$, de $||\kappa_0||_{C^{0,\alpha}}$ et de la norme $C^1$ de $x\mapsto e^{2x}-x$ sur $[-C_0,C_0]$. En utilisant les inégalités d'interpolation classiques (nous renvoyons à \cite[Lemma 6.35]{GT}), pour tout $\epsilon>0$, il existe $C'>0$ (ne dépendant que de $\epsilon$, de $\alpha$ et de la géométrie de $M$) tel que
$$||u||_{C^{0,\alpha}}\leq C'||u||_{C^0}+\epsilon||u||_{C^{2,\alpha}}.$$
En choisissant $\epsilon\leq 1/(2C_2)$ nous trouvons $||u||_{C^{2,\alpha}}\leq C_2(1+C'||u||_{C^0}+1/(2C_2)||u||_{C^{2,\alpha}})$ soit
$$||u||_{C^{2,\alpha}}\leq 2C_2(1+C'||u||_{C^0})\leq 2C_2(1+C'C_0).$$

Appliquons de nouveau le Lemme \ref{l.invertouille}: l'opérateur $L=\Delta-\kappa$ définit un isomorphisme linéaire entre $C^{k+2,\alpha}(M)$ et $C^{k,\alpha}(M)$ et il existe une constante $C_3>0$ ne dépendant que de la géométrie de $M$ et de $||\log(\kappa)||_{C^{k,\alpha}}$ (voir la Remarque \ref{rem.bornes_constantes}) telle que pour tout $u\in C^{k+2,\alpha}(M)$,
$$\|u\|_{C^{k+2,\alpha}}\leq C_3\|Lu\|_{C^{k,\alpha}}.$$
Donc une fonction $u\in C^{k+2,\alpha}(M)$, solution de l'équation \eqref{eq.courbure_u}, doit vérifier
$$\|u\|_{C^{k+2,\alpha}}\leq C_3\left(\|(e^{2u}-u)\kappa-\kappa_0\|_{C^{k,\alpha}}\right)\leq C_{k,\alpha},$$
où $C_{k,\alpha}$ ne dépend que de $k$, $\alpha$, $||u||_{C^{k,\alpha}}$, $||\kappa||_{C^{k,\alpha}}$, $||\kappa_0||_{C^{k,\alpha}}$ et de la norme $C^{k+1}$ de $x\mapsto e^{2x}-x$ sur $[-C_0,C_0]$. Une récurrence immédiate fournit la preuve de la seconde assertion.
\end{proof}

\begin{lemma}\label{l.unicitouille}
Soit $\kappa:M\to(0,\infty)$ une fonction lisse telle que $\log(\kappa)\in C^\infty_b(M)$. Si $u$ et $u'$ sont deux fonctions lisses telles que les métriques $e^{2u}g$ et $e^{2u'}g$ soient complètes et de courbure partout donnée par $-\kappa$. Alors $u=u'$.
\end{lemma}

\begin{proof}[Preuve] Par le Lemme \ref{bornouille} les fonctions $u$ et $u'$ sont éléments de $C^{k+2,\alpha}(M)$ pour toute paire $(k,\alpha)$. De plus ces fonctions satisfont aux équations
$$\Delta^g u=e^{2u}\kappa-\kappa_0,\,\text{et}$$
$$\Delta u'=e^{2u'}\kappa-\kappa_0,$$
de sorte que
$$\Delta(u-u')=e^{\phi}(u-u'),$$
où
$$e^\phi=\int_0^1 2\kappa e^{2tu+2(1-t)u'}dt.$$
En particulier $\phi\in C^\infty_b(M)$. Par le Lemme \ref{l.invertouille} nous avons $u-u'=0$, d'où le résultat.
\end{proof}

\begin{proof}[Preuve du Théorème \ref{th_bounds_elliptic}.] Nous allons prouver ce Théorème en utilisant la méthode de continuité. Soit donc $\kappa:M\to(0,\infty)$ une fonction telle que $\log(\kappa)\in C^{\infty}_b(M)$. Pour tout $t\in[0,1]$, considérons $\kappa_t=(1-t)\kappa_0+t\kappa$. Nous avons clairement $\log(\kappa_t)\in C^\infty_b(M)$ pour tout $t\in[0,1]$. De plus nous avons des bornes pour $\|\log(\kappa_t)\|_{C^{k,\alpha}}$  indépendantes de $t$. Considérons l'ensemble $I\dans[0,1]$ des paramètres $t$ tels qu'il existe une fonction lisse $u_t:M\to\R$ telle que $e^{2u_t}g$ soit une métrique complète de courbure donnée par $-\kappa_t$. Clairement, $I$ contient $0$. De plus, en appliquant le théorème d'Arzela-Ascoli et le Lemme \ref{bornouille} nous voyons que $I$ est fermé.

Prouvons à présent que $I$ est ouvert. Soit $t\in I$ et $u_t:M\to\R$ la fonction lisse telle que $e^{2u_t}g$ soit complète de courbure égale à $-\kappa_t$. Soit l'opérateur non-linéaire
$$K_t v=\Delta v-e^{2v}\kappa_t+\kappa_0,$$
de sorte que $K_t u_t=0$. L'opérateur linéarisé en $u_t$ est
$$LK_t v=\left (\Delta-2e^{u_t}\kappa_t\right) v.$$
Comme la fonction $\phi_t=\log(2)+u_t+\log(\kappa_t)$ appartient à $C^{\infty}_b(M)$, l'opérateur $LK_t=\Delta-e^{\phi_t}$ est un isomorphisme linéaire de $C^{2,\alpha}(M)$ sur $C^{0,\alpha}(M)$. Ainsi par le théorème de la fonction implicite (voir la Remarque \ref{rem.Banach}), pour tout $s$ suffisamment proche de $t$ il existe une fonction $v\in C^{2,\alpha}(M)$ telle que $K_s v=0$. Par régularité elliptique $v$ est lisse. Comme $v$ est bornée, la métrique $e^{2v}g$ est complète, et comme $K_s v=0$, la courbure de cette métrique est précisément $-\kappa_s$. Ceci prouve que $I$ est ouvert. Enfin, l'existence des bornes uniformes désirées vient du Lemme \ref{bornouille}, ce qui achève la preuve du théorème.
\end{proof}

\subsection{Variation lisse et application à certains feuilletages par surfaces hyperboliques}\label{ss.variation_lisse}

\paragraph{\textbf{Variation lisse de la prescription de courbure}}  Par le théorème des fonctions implicites, les solutions à l'équation de prescription de courbure obtenues lors de la preuve du théorème précédent varient de façon lisse avec les données de l'équation. Nous obtenons donc le théorème suivant.

\begin{theorem}[Variation lisse]\label{t.var_lisse} Soit $\D$ le disque unité, $g_0$ une métrique riemannienne lisse et complète sur le produit $\D\times (-1,1)$ et $-\kappa_0$ la fonction qui associe à $(z,s)\in\D\times(-1,1)$ la courbure gaussienne en $z$ de la restriction $g_0|_{\D\times\{s\}}$. Supposons que $\kappa_0$ soit strictement positive et que $\log(\kappa_0)\in C^\infty_b(\D\times(-1,1))$.

Soit $\kappa:\D\times (-1,1)\to(0,\infty)$ une fonction lisse et strictement positive telle que $\log(\kappa)\in C^\infty_b(\D\times(-1,1))$. Alors il existe une unique fonction $u\in C^\infty_b(\D\times(-1,1))$ telle que pour tout $s\in(-1,1)$,  la restriction de la métrique $e^{2u}g$ à $\D\times\{s\}$ soit complète et de courbure gaussienne donnée en tout point par $-\kappa$.
\end{theorem}

\paragraph{\textbf{Feuilletages et tubes normaux}} Soit $(M,\cF)$ une variété feuilletée compacte de codimension $q$. Par compacité de l'espace ambiant il existe un nombre fini de disques de dimension $q$ plongés dans $M$, notés $T_1,\ldots T_n$, qui sont disjoints, transverses à $\cF$ et dont l'union $T=T_1\sqcup T_2\sqcup\ldots\sqcup T_n$ rencontre toute feuille de $\cF$. On appelle $T$ un \emph{système complet de transversales}.

Nous supposons à présent que $(M,\cF)$ soit un feuilletage par surfaces hyperboliques qui possède un \emph{système complet de tubes normaux}, c'est-à-dire qu'il existe une variété $T$ de dimension $q$ ayant un nombre fini de composantes connexes ainsi qu'un difféomorphisme local lisse et surjectif et préservant les feuilletages
$$\sigma:(\D\times T,\cH)\to(M,\cF),$$
$\cH$ étant le feuilletage horizontal de $\D\times T$, vérifiant de plus
\begin{enumerate}
\item pour tout $s\in T$, la restriction de $\sigma$ à $\D\times\{s\}$ est le revêtement universel de la feuille passant par $\sigma(0,s)$;
\item l'image par $\sigma$ de ${\{0\}\times T}$ est un système complet de transversales pour $\cF$.
\end{enumerate}

La notion de \emph{tube normal} apparaît dans \cite{Alcalde_Dalbo_Martinez_Verjovsky} dans le cadre des feuilletages riemanniens \emph{minimaux} (dont toutes les feuilles sont denses). Dans ce cadre toutes les feuilles intersectent une transversale donnée et le système de tubes normaux peut donc se réduire à un unique tube normal. Nous avons adapté cette définition au cas non minimal afin de s'assurer d'avoir une application surjective. Avant de discuter l'existence d'un système complet de tubes normaux nous énonçons le résultat principal de cette sous-section.

\begin{theorem}[Prescription lisse de la courbure]\label{t.prescription_lisse_feuilletage} Soit $(M,\cF)$ un feuilletage compact par surfaces hyperboliques possédant un système complet de tubes normaux et soit $g_0$ une métrique riemannienne sur $M$. Alors pour toute fonction lisse $\kappa:M\to(0,\infty)$ il existe une unique métrique riemannienne lisse conformément équivalente à $g$ dont la courbure gaussienne de la restriction à toute feuille soit en tout point donnée par $-\kappa$.
\end{theorem}

\begin{proof}[Preuve] En utilisant le Théorème de Ghys si besoin (voir le Théorème \ref{negcurvedmetrics}) il est possible de supposer que la métrique $g_0$ induit une métrique à courbure négative (uniformément pincée) dans toute feuille (la preuve donnée en appendice nous fournit une telle métrique lisse dans ce contexte: voir \cite{AY} pour plus de détails). Il ne nous reste plus qu'à appliquer le Théorème \ref{t.var_lisse} au relevé $\tilde g_0$ de $g_0$ au système $\D\times T$ via $\sigma$ et à la fonction $\kappa\circ\sigma:\D\times T\to(0,\infty)$ pour obtenir l'unique métrique lisse dans la classe conforme de $\tilde g_0$ et dont la courbure dans les feuilles de $\cH$ est donnée par $-\kappa\circ\sigma$. Par unicité elle descend en une métrique sur $M$, solution au problème.
\end{proof}

\paragraph{\textbf{Exemples}} Les feuilletages suivants possèdent des systèmes complets de tubes normaux, et par conséquent, on peut leur appliquer le Théorème \ref{t.prescription_lisse_feuilletage}.
\begin{enumerate}
\item Les feuilletages obtenus par \emph{suspension} d'une action d'un groupe de surface de genre $\geq 2$ sur une variété compacte (voir \cite{Godbillon}).
\item Les feuilletages \emph{transversalement de Lie}, qui sont définis par des submersions locales sur un groupe de Lie, les changements de cartes transverses étant des translations de ce groupe de Lie. Plus généralement ce théorème s'applique à tout feuilletage défini par fibration équivariante (voir \cite[\S II.1.4]{HHA}).
\item Les feuilletages \emph{riemanniens}, définis par des submersions locales sur une variété riemannienne, dont les changements de cartes sont transversalement des isométries locales. Ces feuilletages sont décrits par le \emph{théorème de structure de Molino}: voir \cite{Molino,Molino_livre}.
\end{enumerate}

Ces feuilletages ont en commun la propriété suivante. Leurs groupoïdes d'homotopie $\Pi_1(\cF)$ (voir la Remarque \ref{rem_groupoide}) sont localement triviaux: voir \cite[\S 2.A]{Alcalde_groupoides} pour le cas des feuilletages de Lie, et \cite[\S 2.C]{Alcalde_groupoides} pour le cas des feuilletages riemannien.

Le fait que tout feuilletage dont le groupoïde d'homotopie est localement trivial possède un système complet de tubes normaux s'obtient en copiant verbatim la preuve de \cite[Proposition 3]{Alcalde_Dalbo_Martinez_Verjovsky}  (faite dans le cas minimal).

Nous donnons ci-dessous trois exemples intéressants de feuilletages auxquels s'applique ce théorème.
\begin{itemize}
\item La suspension d'une représentation d'un groupe de surface hyperbolique sur un sous-groupe libre de rotations de la sphère de dimension $2$. C'est un exemple de feuilletage riemannien dont les feuilles sont quasi-isométriques à des arbres. Dans la même veine, \cite{BGSS} donne des représentations fidèles de n'importe quel groupe de surface hyperbolique dans le groupe des rotations $SO(3)$ dont les images sont denses. Par suspension, on obtient des feuilletages riemanniens et minimaux dont les feuilles sont simplement connexes.
\item Le feuilletage centre-stable du flot géodésique sur le fibré unitaire tangent d'une surface hyperbolique, dont les feuilles sont des disques et des cylindres hyperboliques.
\item Le feuilletage centre-stable de la suspension d'un automorphisme linéaire et hyperbolique du tore de dimension $2$, qui est un feuilletage transversalement affine sur une solvariété de dimension $3$ dont les feuilles sont des disques et des cylindres hyperboliques.
\end{itemize}

\section{Appendices}

\subsection{Sur la topologie de Cheeger-Gromov} Ci-dessous nous prouvons les résultats clés que nous avons utilisés dans cet article, à savoir le Lemme \ref{equivalence_subtile} et le Théorème \ref{CGHausdorff}. Nous commençons la preuve de ces résultats par un lemme élémentaire.

\begin{lemma}\label{l.compact_ouvert}
Soient $M$ et $N$ deux variétés de dimension finie et $\Phi_n:M\to N$ une suite d'homéomorphismes sur leurs images convergeant pour la topologie compacte-ouverte vers $\Phi:M\to N$, un homéomorphisme sur son image.

Alors pour tout compact $K\dans\Iim(\Phi)$ il existe $n_0\in\N$ tel que pour tout $n\geq n_0$, $K\dans\Iim(\Phi_n)$. En particulier la suite des restrictions $(\Phi^{-1}_n|_K)_{n\geq n_0}$ converge vers $\Phi^{-1}|_K$ dans la topologie compacte-ouverte.
\end{lemma}

\begin{proof}[Preuve]
Soient $y\in K$ et $x=\Phi^{-1}(y)\in M$. Soit $\Psi:(\R^d,0)\to (N,y)$ un homéomorphisme sur un ouvert de $N$ contenant $y$. Pour tout $r>0$, notons $B_r$ l'image par $\Psi$ de la boule centrée en $0$ de rayon $r$, et $\overline{B_r}$, son adhérence. Soit $r>0$ vérifiant $B_{2r}\dans\Iim(\Phi)$ et soit $X_r=\Phi^{-1}(\overline{B_r})$. Puisque $X_r$ et $\partial X_r$ sont tous deux compacts, nous avons, pour $n$ suffisamment grand, $\Phi_n(x)\in B_{r/2}$, $\Phi_n(X_r)\dans B_{3r/2}$ et $\Phi_n(\partial X_r)\dans B_{3r/2}\moins B_{r/2}$. Puisque $\Phi_n$ est un homéomorphisme sur son image, il vient $B_{r/2}\dans\Iim(\Phi_n)$. La première propriété est alors conséquence de la compacité de $K$.
\end{proof}

\begin{corollary}\label{conv_des_inverses}
Soient $M$ et $N$ deux variétés de dimension finie et $\Phi_n:M\to N$ une suite de difféomorphismes lisses sur leurs images convergeant dans la topologie $\Ciloc$ vers $\Phi:M\to N$, un difféomorphisme lisse sur son image.

Soit $\Omega\dans N$ un ouvert relativement compact inclus dans $\Iim(\Phi)$ et soit $n_0=n_0(\overline{\Omega})$ l'entier obtenu dans le Lemme \ref{l.compact_ouvert}. Alors la suite $(\Phi_n^{-1}|_{\Omega})_{n\geq n_0}$  converge vers $\Phi^{-1}|_\Omega$ dans la topologie $\Ciloc$.
\end{corollary}

\begin{proof}[Preuve]
C'est une conséquence triviale du Lemme \ref{l.compact_ouvert} et de la ``chain rule''.
\end{proof}

Nous pouvons alors prouver le Lemme \ref{equivalence_subtile}, dont nous reproduisons ci-dessous l'énoncé par souci de lisibilité.

\begin{lemma}
Soient $M$ et $N$ deux variétés lisses de même dimension. Soient $(g_n)$ et $(h_n)$ deux suites de métriques riemanniennes sur $M$ et $N$ respectivement, qui convergent vers deux métriques $g$ et $h$ dans la topologie  $\Ciloc$. Soit $K\dans N$ une partie compacte et $\Phi_n:M\to K$ une suite d'applications lisses telles que pour tout $n\in\N$, $\Phi_n^\ast h_n=g_n$. Alors la suite $(\Phi_n)$ possède une sous-suite convergeant vers une application $\Phi:M\to K$ au sens $\Ciloc$ vérifiant de plus $\Phi^\ast h=g$.
\end{lemma}

\begin{proof}[Preuve] Premièrement, pour tous $x\in M$, $n\in\N$ et $r>0$ nous notons respectivement $V_n(x,r)$ et $V(x,r)$ les $r$-voisinages de $x$ pour les métriques $g_n$ et $g$ respectivement, ainsi que $B_n(x,r)$ et $B(x,r)$, les boules de rayon $r$ centrées en l'origine dans l'espace tangent $T_xM$ pour ces mêmes métriques.

Soit $X\dans M$ un ensemble dénombrable et dense. Par compacité et un argument diagonal, il existe une application $\Phi:X\to K$ vers laquelle converge ponctuellement une sous-suite de $\Phi_n|_X:X\to K$. Soit $\Omega$ un ouvert relativement compact de $M$. Par compacité il existe $r>0$ strictement inférieur aux rayons d'injectivité de tout point de $K$ pour toutes les métriques $h_n$ et $h$. Soit maintenant $x\in\Omega\cap X$. Pour tout $n\in\N$ nous notons $E_{g_n}=\exp_{g_n,x}$, $E_{g}=\exp_{g,x}$, $E_{h_n}=\exp_{h_n,\Phi_n(x)}$ et $E_{h}=\exp_{h,\Phi(x)}$ les applications exponentielles définies sur les boules de rayon $r$ des espaces tangents correspondant. Alors nous avons pour tout $n$ et $v\in B_n(x,r)$
$$\left(\Phi_n\circ E_{g_n}\right)(v)=\left(E_{h_n}\circ D\Phi_n(x)\right)(v).$$
Nous en déduisons que pour tout $y\in V_n(x,r)$
$$\Phi_n(y)=\left(E_{h_n}\circ D\Phi_n(x)\circ E_{g_n}^{-1}\right)(y).$$

Par le Lemme \ref{l.compact_ouvert} il existe $n_0$ tel que pour tout $n\geq n_0$ , $V(x,r/2)\dans E_{g_n}(B_n(x,r))$. De plus, le Corollaire \ref{conv_des_inverses} entraîne que la suite des restrictions $(E_{g_n}^{-1}|_{V(x,r/2)})_{n\geq n_0}$ converge vers $E_g^{-1}$ dans la topologie $\Ciloc$. D'autre part la suite d'applications linéaires $D\Phi_n(x)$ est bornée et donc possède une sous-suite qui converge vers une application linéaire notée $A$. Enfin, puisque $\Phi_n(x)$ converge vers $\Phi(x)$, la suite $E_{h_n}$ converge vers $E_h$ dans la topologie $\Ciloc$.

Nous déduisons de tout ceci que la suite $\Phi_n$ converge sur $V(x,r/2)$ vers $E_h\circ A\circ E_g'$ au sens $\Ciloc$, et nous obtenons le résultat désiré par un argument diagonal.
\end{proof}

Nous déduisons le Théorème \ref{CGHausdorff}, qui établit la propriété de Hausdorff de la convergence au sens de Cheeger-Gromov.

\begin{theorem}Soit $(M_n,g_n,p_n)$ une suite de variétés riemanniennes pointées qui converge vers $(M,g,p)$ et $(M',g',p')$ dans la topologie de Cheeger-Gromov. Alors il existe un difféomorphisme $\Phi:M\to M'$ tel que $\Phi^\ast g'=g$ et $\Phi(p)=p'$.
\end{theorem}

\begin{proof}[Preuve]
Soient $(F_n)$ et $(F_n')$ des suites d'applications de convergence de $(M_n,g_n,p_n)$ vers $(M,g,p)$ et $(M',g',p')$ respectivement. Fixons $r>0$ et soit $V$ le $r$-voisinage de $p'\in M'$ pour $g'$. Soit $n_1\in\N$ tel que pour tout $n\geq n_1$ l'application $F_n'$ soit un difféomorphisme de $V$ sur son image. Quitte à augmenter $n_1$ nous pouvons supposer que $F_n'(V)$ est inclus dans l'image de $F_n$ pour tout $n\geq n_1$, et que $(F_n^{-1}\circ F_n')(V)$ contient le $r/2$-voisinage de $p\in M$ pour $g$. Nous noterons $U$ ce voisinage.

Rappelons que si $(N,h)$ est une variété riemannienne et $k\geq 0$ nous notons $\|.\|_{C^k(N,h)}$ la norme $C^k$ des sections de fibrés sur $N$ par rapport à $h$. Par hypothèse
$$\lim_{n\to\infty}\left\|(F_n')^\ast g_n-g'\right\|_{C^k(V,g')}=0.$$
Nous en déduisons successivement que
$$\lim_{n\to\infty}\left\|(F_n')^\ast g_n-g'\right\|_{C^k(V,(F_n')^\ast g_n)}=0,$$
$$\lim_{n\to\infty}\left\|g_n-(F_n')_\ast g'\right\|_{C^k(F_n'(V),g_n)}=0,$$
et enfin
$$\lim_{n\to\infty}\left\|(F_n)^\ast g_n-(F_n)^\ast (F_n')_\ast g'\right\|_{C^k(U,(F_n)^\ast g_n)}=0.$$
Or, par hypothèse,
\begin{equation}\label{eqhypouille}
\lim_{n\to\infty}\left\|(F_n)^\ast g_n-g\right\|_{C^k(U,g)}=0,
\end{equation}
d'où
\begin{equation}\label{conseqeshypouille}
\lim_{n\to\infty}\left\|(F_n)^\ast g_n-(F_n)^\ast (F_n')_\ast g'\right\|_{C^k(U,g)}=0.
\end{equation}
Il découle de \eqref{eqhypouille}, de \eqref{conseqeshypouille} et de l'inégalité triangulaire, que $((F_n')^{-1}\circ F_n)^\ast g'$ converge dans la topologie $C^\infty$ vers $g$.

Le lemme précédent implique que la suite $\Phi_n=(F_n')^{-1}\circ F_n$, $n\geq n_1$ possède une sous-suite qui converge dans la topologie $\Ciloc$ vers une application $\Phi:U\to M'$ qui est un difféomorphisme sur son image et telle que $\Phi^\ast g'=g$ sur $U$ et vérifiant de plus $\Phi(p)=p'$. Par un argument diagonal nous pouvons supposer que $\Phi$ est définie sur la variété $M$ toute entière. De plus $\Phi^\ast g'=g$ donc $\Phi$ est une isométrie locale. Comme $M$ et $M'$ sont complètes $\Phi$ est surjective. C'est donc un difféomorphisme entre $M$ et $M'$. 
\end{proof}

\subsection{Gauss-Bonnet et existence de métriques laminées à courbure négative}\label{ss.gaussboubouille}

Nous donnons ici un argument alternatif pour l'existence de métriques laminées à courbure négative indépendant du théorème de Candel. L'argument que nous présentons est dû à Ghys et apparaît dans \cite[Theorem C]{AY}. Nous l'incluons ici car il nous paraît simple et agréable.

\begin{theorem}
\label{negcurvedmetrics}
Soit $(X,\cL)$ une lamination compacte par surfaces hyperboliques et $g$ une métrique laminée. Alors il existe $u\in C^{\infty}_{l}(X)$ telle que la courbure gaussienne dans les feuilles pour la métrique $e^{2u}g$ soit partout pincée entre deux constantes négatives.
\end{theorem}

Nous considérons l'\emph{opérateur de Laplace-Beltrami laminé} défini pour toute fonction $u\in C^2_{l}(M)$ par $\Delta^{\cL}u(x)=\Delta_{L_x}u(x)$, où $x\in X$ et $\Delta_{L_x}$ est l'opérateur de Laplace-Beltrami sur la feuille $L_x$ passant par $x$.

\begin{defn}
\label{harmonic}
Une \emph{mesure harmonique} d'une lamination riemannienne compacte $(X,\cL)$ est une mesure de probabilité $m$ sur $X$ telle que pour tout $u\in C^2_{l}(X)$
$$\int_M\Delta^{\cL}u\,dm=0.$$
\end{defn}

L'existence de telles mesures pour les laminations compactes découle des travaux de Garnett (cf. \cite{Garnett}).

\begin{rem}
Par compacité de $X$, l'espace $C^{\infty}_{l}(X)$ est dense dans $C^2_{l}(X)$ (cf. le Lemme \ref{smoothdense}). Donc une mesure $m$ est harmonique si et seulement si pour tout $u\in C_{l}^{\infty}(X)$ l'intégrale de $\Delta^{\cL} u$ contre $m$ est nulle.
\end{rem}

Nous utiliserons une version laminée du Théorème de Gauss-Bonnet dû à Ghys (voir \cite{Ghys_GB}). Ce résultat est plus faible que le Théorème d'uniformisation simultanée et n'utilise en fait que la continuité supérieure des applications d'uniformisation des feuilles.

\begin{theorem}[Ghys]
\label{GrosBonnet}
Soit $(X,\cL)$ une lamination compacte par surfaces hyperboliques et $g$ une métrique laminée. Soit $\kappa(x)$ denote la courbure gaussienne au point $x\in X$ de la feuille passant par $x$. Alors
$$\int_X\kappa\,dm<0.$$
\end{theorem}

Rappelons que si $g'=e^{2u}g$, les courbures $\kappa'(x)$ et $\kappa(x)$ en $x$ des métriques $g'$ et $g$ sont reliées par la formule
$$\kappa'(x)=e^{-2u(x)}\left(\kappa(x)-\Delta u(x)\right).$$

Ainsi la preuve du Théorème \ref{negcurvedmetrics} se réduit à celle du Lemme suivant dû à Ghys (see \cite{AY,Ghys_GB,Ghys_parabolique}).

\begin{lemma}
\label{keylemmaGhys}
Soit $(X,\cL)$ une lamination compacte par surfaces hyperboliques et $g$ une métrique laminée. Alors il existe $u\in C^{\infty}_{l}(X)$ telle que pour tout $x\in X$
$$\kappa(x)-\Delta^{\cL}u(x)<0.$$
\end{lemma}

\begin{proof}[Preuve]
Définissons l'espace de Banach $\cC$ formé des fonctions continues sur $X$ muni de la norme de la convergence uniforme, ainsi que le sous-espace fermé $\cH\dans\cC$, défini comme l'adhérence de $\{\Delta^{\cL}u;\,u\in C^{\infty}_l(X)\}$. L'espace $\cH$ étant fermé, la projection $\Pi:\cC\to\cC/\cH$ est continue et ouverte entre espaces de Banach. Remarquons que $(\cC/\cH)'$ s'identifie isométriquement au complément orthogonal de $\cH$ (i.e. l'espace des formes linéaires continues sur $\cC$ s'annulant sur $\cH$). Soit $\Lambda^-\dans\cC$ le cône des fonctions strictement négatives. La projection $\widehat{\Lambda}^-=\Pi(\Lambda^-)$ est ouverte et convexe. 
La conclusion du lemme est équivalente au fait que $\widehat{\kappa}=\Pi(\kappa)$ appartienne à $\widehat{\Lambda}^-$.

Supposons le contraire. Par le théorème de Hahn-Banach (voir \cite[Lemme 1.3]{Brezis}), il existe une forme linéaire continue $m\in(\cC/\cH)'$ telle que pour tout $u\in\widehat{\Lambda}^-$, $m(u)< a=m(\widehat{\kappa})$. En évaluant $m$ sur les fonctions de la forme $\lambda u$, $\lambda>0$ nous obtenons $m\leq 0$ sur $\widehat{\Lambda}^-$ (en faisant $\lambda$ tendre vers l'infini) et $a\geq 0$ (cette fois en faisant tendre $\lambda$ vers zéro). Ceci contredit le Théorème de Gauss-Bonnet de Ghys, car $m$ s'identifie alors à un élément \emph{positif} du complément orthogonal de $\cH$, c'est-à-dire (par le Théorème de représentation de Riesz) à une mesure harmonique vérifiant $\int_X\kappa\,dm=a\geq 0$. 
\end{proof}

\bibliographystyle{plain}

\begin{thebibliography}{10}

\bibitem{Ahlfors2006}
L.~Ahlfors.
\newblock {\em Lectures on quasiconformal mappings}, volume~38 of {\em
  University Lecture Series}.
\newblock American Mathematical Society, Providence, RI, second edition, 2006.
\newblock With supplemental chapters by C. J. Earle, I. Kra, M. Shishikura and
  J. H. Hubbard.

\bibitem{AhBe}
L.~Ahlfors and L.~Bers.
\newblock Riemann's mapping theorem for variable metrics.
\newblock {\em Ann. of Math. (2)}, 72:385--404, 1960.

\bibitem{Alcalde_groupoides}
F.~Alcalde~Cuesta.
\newblock Groupo\"{\i}de d'homotopie d'un feuilletage riemannien et
  r\'{e}alisation symplectique de certaines vari\'{e}t\'{e}s de {P}oisson.
\newblock {\em Publ. Mat.}, 33(3):395--410, 1989.

\bibitem{Alcalde_Dalbo_Martinez_Verjovsky}
F.~Alcalde~Cuesta, F.~Dal'Bo, M.~Mart\'{\i}nez, and A.~Verjovsky.
\newblock Unique ergodicity of the horocycle flow on {R}iemannnian foliations.
\newblock {\em Ergodic Theory Dynam. Systems}, 40(6):1459--1479, 2020.

\bibitem{Alcalde_Hector}
F.~Alcalde~Cuesta and G.~Hector.
\newblock Feuilletages en surfaces, cycles \'{e}vanouissants et
  vari\'{e}t\'{e}s de {P}oisson.
\newblock {\em Monatsh. Math.}, 124(3):191--213, 1997.

\bibitem{Alvharmonic}
S.~Alvarez.
\newblock Harmonic measures and the foliated geodesic flow for foliations with
  negatively curved leaves.
\newblock {\em Ergodic Theory Dynam. Systems}, 36(2):355--374, 2016.

\bibitem{Alvarezu-Gibbs}
S.~Alvarez.
\newblock Gibbs {$u$}-states for the foliated geodesic flow and transverse
  invariant measures.
\newblock {\em Israel J. Math.}, 221(2):869--940, 2017.

\bibitem{AlvarezGibbs}
S.~Alvarez.
\newblock Gibbs measures for foliated bundles with negatively curved leaves.
\newblock {\em Ergodic Theory Dynam. Systems}, 38(4):1238--1288, 2018.

\bibitem{AB}
S.~Alvarez, J.~Brum.
\newblock Topology of leaves for minimal laminations by non-simply connected hyperbolic surfaces.
\newblock {\em Groups Geom. Dyn.}, 2020.

\bibitem{ABMP}
S.~Alvarez, J.~Brum, M.~Mart{\'i}nez, and R.~Potrie.
\newblock Topology of leaves for minimal laminations by hyperbolic surfaces.
\newblock {\em Preprint, [arXiv:1906.10029]}, 2019.

\bibitem{AL}
S.~Alvarez and P.~Lessa.
\newblock The {T}eichm\"{u}ller space of the {H}irsch foliation.
\newblock {\em Ann. Inst. Fourier (Grenoble)}, 68(1):1--51, 2018.

\bibitem{AlvarezSmith}
S.~Alvarez and G.~Smith.
\newblock Earthquakes and graftings of hyperbolic surface laminations.
\newblock {\em Int. Math. Res. Not. (to appear)}, 2020.

\bibitem{AY}
S.~Alvarez and J.~Yang.
\newblock Physical measures for the geodesic flow tangent to a transversally
  conformal foliation.
\newblock {\em Ann. Inst. H. Poincar\'{e} Anal. Non Lin\'{e}aire},
  36(1):27--51, 2019.

\bibitem{AlvarezBarral}
J.~\'Alvarez~L\'opez and R.~Barral~Lij{\'o}.
\newblock Realization of manifolds as leaves using graph colorings.
\newblock {\em Preprint, arXiv:2002.08662}, 2020.

\bibitem{ABC}
J.~\'{A}lvarez L\'{o}pez, R.~Barral~Lij\'{o}, and A.~Candel.
\newblock A universal {R}iemannian foliated space.
\newblock {\em Topology Appl.}, 198:47--85, 2016.

\bibitem{AlvLopCandel}
J.~\'{A}lvarez L\'{o}pez and A.~Candel.
\newblock {\em Generic coarse geometry of leaves}, volume 2223 of {\em Lecture
  Notes in Mathematics}.
\newblock Springer, Cham, 2018.

\bibitem{AMc}
P.~Aviles and R.~McOwen.
\newblock Conformal deformations of complete manifolds with negative curvature.
\newblock {\em J. Differential Geometry}, 21(2):269--281, 1985.

\bibitem{BGS}
W.~Ballmann, M.~Gromov, and V.~Schroeder.
\newblock {\em Manifolds of nonpositive curvature}, volume~61 of {\em Progress
  in Mathematics}.
\newblock Birkh\"{a}user Boston, Inc., Boston, MA, 1985.

\bibitem{Berger}
M.~Berger.
\newblock Riemannian structures of prescribed {G}aussian curvature for compact
  {$2$}-manifolds.
\newblock {\em J. Differential Geometry}, 5:325--332, 1971.

\bibitem{BK}
J.~Bland and M.~Kalka.
\newblock Complete metrics conformal to the hyperbolic disc.
\newblock {\em Proc. Amer. Math. Soc.}, 97(1):128--132, 1986.

\bibitem{BGM}
C.~Bonatti, X.~G\'{o}mez-Mont, and M.~Mart\'{\i}nez.
\newblock Foliated hyperbolicity and foliations with hyperbolic leaves.
\newblock {\em Ergodic Theory Dynam. Systems}, 40(4):881--903, 2020.

\bibitem{BGSS}
E.~Breuillard, T.~Gelander, J.~Souto, and P.~Storm.
\newblock Dense embeddings of surface groups
\newblock {\em Geom. Topol.}, 10:1376--1389, 2006.

\bibitem{Brezis}
H.~Brezis.
\newblock {\em Analyse fonctionnelle}.
\newblock {C}ollection {M}ath\'ematiques {A}ppliqu\'ees pour la {M}a\^\i trise.
  Masson, Paris, 1983.
\newblock Th{\'e}orie et applications.

\bibitem{Calegari_book}
D.~Calegari.
\newblock {\em Foliations and the geometry of 3-manifolds}.
\newblock Oxford Mathematical Monographs. Oxford University Press, Oxford,
  2007.

\bibitem{CLN}
C.~Camacho and A.~Lins~Neto.
\newblock {\em Geometric theory of foliations}.
\newblock Birkh\"auser Boston, Inc., Boston, MA, 1985.

\bibitem{Candel}
A.~Candel.
\newblock Uniformization of surface laminations.
\newblock {\em Ann. Sci. \'ENS. (4)}, 26(4):489--516, 1993.

\bibitem{Candel_harmonic}
A.~Candel.
\newblock The harmonic measures of {L}ucy {G}arnett.
\newblock {\em Adv. Math.}, 176(2):187--247, 2003.

\bibitem{Chavel}
I.~Chavel.
\newblock {\em Riemannian geometry---a modern introduction}, volume 108 of {\em
  Cambridge Tracts in Mathematics}.
\newblock Cambridge University Press, Cambridge, 1993.

\bibitem{Ch}
J.~Cheeger.
\newblock Finiteness theorems for {R}iemannian manifolds.
\newblock {\em Amer. J. Math.}, 92:61--74, 1970.

\bibitem{ChowKnopf}
B.~Chow and D.~Knopf.
\newblock {\em The {R}icci flow: an introduction}, volume 110 of {\em
  Mathematical Surveys and Monographs}.
\newblock American Mathematical Society, Providence, RI, 2004.

\bibitem{deroin2007}
B.~Deroin.
\newblock Nonrigidity of hyperbolic surfaces laminations.
\newblock {\em Proc. Amer. Math. Soc.}, 135(3):873--881, 2007.

\bibitem{Garnett}
L.~Garnett.
\newblock Foliations, the ergodic theorem and {B}rownian motion.
\newblock {\em J. Funct. Anal.}, 51(3):285--311, 1983.

\bibitem{Ghys_GB}
{\'E}.~Ghys.
\newblock Gauss-{B}onnet theorem for {$2$}-dimensional foliations.
\newblock {\em J. Funct. Anal.}, 77(1):51--59, 1988.

\bibitem{Ghys_parabolique}
{\'E}.~Ghys.
\newblock Sur l'uniformisation des laminations paraboliques.
\newblock In {\em Integrable systems and foliations/{F}euilletages et
  syst\`emes int\'egrables ({M}ontpellier, 1995)}, volume 145 of {\em Progr.
  Math.}, pages 73--91. Birkh\"auser Boston, Boston, MA, 1997.
  
\bibitem{Ghys_Laminations}
{\'E}.~Ghys.
\newblock Laminations par surfaces de {R}iemann.
\newblock In {\em Dynamique et g\'eom\'etrie complexes ({L}yon, 1997)},
  volume~8 of {\em Panor. Synth\`eses}, pages ix, xi, 49--95. Soc. Math.
  France, Paris, 1999.

\bibitem{GT}
D.~Gilbarg and N.~Trudinger.
\newblock {\em Elliptic partial differential equations of second order}, volume
  224 of {\em Grundlehren der Mathematischen Wissenschaften [Fundamental
  Principles of Mathematical Sciences]}.
\newblock Springer-Verlag, Berlin, second edition, 1983.

\bibitem{Godbillon}
C.~Godbillon.
\newblock {\em Feuilletages}, volume~98 of {\em Progress in Mathematics}.
\newblock Birkh\"{a}user Verlag, Basel, 1991.
\newblock \'{E}tudes g\'{e}om\'{e}triques.

\bibitem{Gr}
M.~Gromov.
\newblock {\em Metric structures for {R}iemannian and non-{R}iemannian spaces},
  volume 152 of {\em Progress in Mathematics}.
\newblock Birkh\"auser Boston, Inc., Boston, MA, 1999.
\newblock With appendices by M. Katz, P. Pansu and S. Semmes.

\bibitem{HHA}
G.~Hector and U.~Hirsch.
\newblock {\em Introduction to the geometry of foliations. {P}art {A}},
  volume~1 of {\em Aspects of Mathematics}.
\newblock Friedr. Vieweg \& Sohn, Braunschweig, 1981.

\bibitem{HT}
D.~Hulin and M.~Troyanov.
\newblock Prescribing curvature on open surfaces.
\newblock {\em Math. Ann.}, 293(2):277--315, 1992.

\bibitem{KW}
J.L. Kazdan and F.W. Warner.
\newblock Curvature functions for compact {$2$}-manifolds.
\newblock {\em Ann. of Math. (2)}, 99:14--47, 1974.

\bibitem{KW2}
J.L. Kazdan and F.W. Warner.
\newblock Curvature functions for open {$2$}-manifolds.
\newblock {\em Ann. of Math. (2)}, 99:203--219, 1974.

\bibitem{Lessa}
P.~Lessa.
\newblock Reeb stability and the {G}romov-{H}ausdorff limits of leaves in
  compact foliations.
\newblock {\em Asian J. Math.}, 19(3):433--463, 2015.

\bibitem{Lessa2}
P.~Lessa.
\newblock Brownian motion on stationary random manifolds.
\newblock {\em Stoch. Dyn.}, 16(2):1660001, 66, 2016.

\bibitem{Molino}
P.~Molino.
\newblock G\'{e}om\'{e}trie globale des feuilletages riemanniens.
\newblock {\em Nederl. Akad. Wetensch. Indag. Math.}, 44(1):45--76, 1982.

\bibitem{Molino_livre}
P.~Molino.
\newblock {\em Riemannian foliations}, volume~73 of {\em Progress in
  Mathematics}.
\newblock Birkh\"{a}user Boston, Inc., Boston, MA, 1988.
\newblock Translated from the French by Grant Cairns, With appendices by
  Cairns, Y. Carri\`ere, \'{E}. Ghys, E. Salem and V. Sergiescu.

\bibitem{MS}
C.C. Moore and C.L. Schochet.
\newblock {\em Global analysis on foliated spaces}, volume~9 of {\em
  Mathematical Sciences Research Institute Publications}.
\newblock Cambridge University Press, New York, second edition, 2006.

\bibitem{MunizVerjovsky}
R.~Mu\~niz and A.~Verjovsky.
\newblock Uniformization of compact foliated spaces by surfaces of hyperbolic
  type via the {R}icci flow.
\newblock {\em Preprint}, 2020.

\bibitem{Osserman}
R.~Osserman.
\newblock On the inequality {$\Delta u\geq f(u)$}.
\newblock {\em Pacific J. Math.}, 7:1641--1647, 1957.

\bibitem{penner-saric2008}
R.C. Penner and D.~{\v{S}}ari{\'c}.
\newblock Teichm\"uller theory of the punctured solenoid.
\newblock {\em Geom. Dedicata}, 132:179--212, 2008.

\bibitem{Pet}
P.~Petersen.
\newblock {\em Riemannian geometry}, volume 171 of {\em Graduate Texts in
  Mathematics}.
\newblock Springer, New York, second edition, 2006.

\bibitem{Poinc}
H.~Poincar\'e.
\newblock Les fonctions fuchsiennes et l'\'equation ${\Delta} u=e^u$.
\newblock {\em J. Math. Pures Appl. (5\textsuperscript{\`e} s\'erie)},
  293(5):137--230, 1898.

\bibitem{saric2009}
D.~{\v{S}}ari{\'c}.
\newblock The {T}eichm\"uller theory of the solenoid.
\newblock In {\em Handbook of {T}eichm\"uller theory. {V}ol. {II}}, volume~13
  of {\em IRMA Lect. Math. Theor. Phys.}, pages 811--857. Eur. Math. Soc.,
  Z\"urich, 2009.

\bibitem{Smith_Plateau}
G.~Smith.
\newblock Hyperbolic {P}lateau problems.
\newblock {\em Geom. Dedicata}, 176:31--44, 2015.

\bibitem{sullivan1988}
D.~Sullivan.
\newblock Bounds, quadratic differentials, and renormalization conjectures.
\newblock In {\em American {M}athematical {S}ociety centennial publications,
  {V}ol. {II} ({P}rovidence, {RI}, 1988)}, pages 417--466. Amer. Math. Soc.,
  Providence, RI, 1992.

\bibitem{sullivan1993}
D.~Sullivan.
\newblock Linking the universalities of {M}ilnor-{T}hurston, {F}eigenbaum and
  {A}hlfors-{B}ers.
\newblock In {\em Topological methods in modern mathematics ({S}tony {B}rook,
  {NY}, 1991)}, pages 543--564. Publish or Perish, Houston, TX, 1993.

\bibitem{Ver}
A.~Verjovsky.
\newblock A uniformization theorem for holomorphic foliations.
\newblock In {\em The {L}efschetz centennial conference, {P}art {III} ({M}exico
  {C}ity, 1984)}, volume~58 of {\em Contemp. Math.}, pages 233--253. Amer.
  Math. Soc., Providence, RI, 1987.

\end{thebibliography}

\end{document}